\documentclass[12pt,reqno]{article}
\usepackage[usenames]{color}
\usepackage{amssymb}
\usepackage{graphicx}
\usepackage{amscd}
\usepackage[colorlinks=true,
linkcolor=webgreen,
filecolor=webbrown,
citecolor=webgreen,
pdfstartpage=1]{hyperref}

\definecolor{webgreen}{rgb}{0,.5,0}
\definecolor{webbrown}{rgb}{.6,0,0}

\usepackage{color}
\usepackage{float}

\usepackage{graphics,amsmath,amssymb}
\usepackage{amsthm}
\usepackage{amsfonts}
\usepackage{latexsym}

\newcommand{\seqnum}[1]{\href{http://oeis.org/#1}{\ul{#1}}}

\newcommand{\ul}{\underline}
\newcommand{\ol}{\overline}

\begin{document}

\theoremstyle{plain}
\newtheorem{theorem}{Theorem}
\newtheorem{corollary}[theorem]{Corollary}
\newtheorem{lemma}[theorem]{Lemma}
\newtheorem{proposition}[theorem]{Proposition}
\newtheorem{obs}[theorem]{Observation}

\theoremstyle{definition}
\newtheorem{definition}[theorem]{Definition}
\newtheorem{example}[theorem]{Example}
\newtheorem{conjecture}[theorem]{Conjecture}

\theoremstyle{definition}
\newtheorem{remark}[theorem]{Remark}
\begin{center}
\vskip 1cm

{\LARGE\bf Castling tree of tight Dyck nests with applications to odd and middle-levels graphs}
\vskip 1cm
\large
Italo J. Dejter\\
University of Puerto Rico\\
Rio Piedras, PR 00936-8377\\
\href{mailto:italo.dejter@gmail.com}{\tt italo.dejter@gmail.com} \\
\end{center}

\begin{abstract}\noindent A subfamily of Dyck words called tight Dyck words is seen to correspond, via a ``castling'' procedure, to the vertex set of an ordered tree $T$. From $T$, a ``blowing'' operation recreates the whole family ol Dyck words. The vertices of $T$ can be elementarily  updated all along $T$. This simplifies an edge-supplementary arc-factorization view of Hamilton cycles of odd and middle-levels graphs found by T. M\"utze et al. This takes into account that the Dyck words represent: {\bf(a)} the cyclic and dihedral vertex classes of odd and middle-levels graphs, respectively, and {\bf(b)} the cycles of their 2-factors, as found by T. M\"utze et al. 
\end{abstract}

\section{Introduction}\label{section1}
 
For $0<k<n\in\mathbb{Z}$, the {\it Kneser graph} $K_{n,k}$ \cite{Hcs,namrata} has as its vertices all $k$-subsets of the set $[0,n-1]=\{0,1,\ldots,n-1\}$, with any two such vertices adjacent if and only if their intersection as $k$-subsets of $[0,n-1]$ is empty. For $n=2k+1$, $K(n,k)$ is the {\it odd graph} $O_k$. We represent the vertices of $ O_k$ by the
{\it characteristic vectors} of those $k$-subsets of $[0,n-1]$, so that each such $k$-subset is the {\it support} of its characteristic vector.

Those vectors represent the members of the $k$-level $L_k$ of the Boolean lattice $B_{2k+1}$ on $[0,2k]$. The complements of the reversed strings of those vectors are taken to represent the $(k+1)$-level $L_{k+1}$ of $B_{2k+1}$. The union of these two levels, $L_k\cup L_{k+1}$, is the vertex set of the {\it middle-levels graph} $M_k$~\cite{Havel}, with adjacency given by the inclusion of $ L_k$ in $L_{k+1}$. This yields a 2-covering graph map $\Psi_k:M_k\rightarrow O_k$ expressible via the {\it reversed complement} bijection $\aleph:V(M_k)\rightarrow V(M_k)$ given by $\aleph(v_0v_1\cdots v_{2k-1}v_{2k})=\bar{v}_{2k}\bar{v}_{2k-1}\cdots\bar{v}_1\bar{v}_0$, where $\bar{0}=1$ and $\bar{1}=0$, so that $\aleph(L_k)=L_{k+1}$ and $\aleph(L_{k+1})=L_k$, and extending to a graph automorphism of $M_k$, again denoted w.l.o.g. by $\aleph$.
In fact, $\Psi:M_k\rightarrow O_k$ is characterized by its restriction to the identity map over $L_k$ and by its restriction to $\aleph$ over $L_{k+1}$.
  
Odd graphs were shown to be Hamiltonian in \cite{Hcs}, as an initial step in proving that sparse Kneser graphs were Hamiltonian. And this result was extended in \cite{namrata} to all Kneser graphs.

 More specifically,
an investigation of relations of Dyck words~\cite{Hcs} controlled by the infinite ordered tree $T$ of {\it restricted-growth strings} (~\cite[p. 325]{Arndt}) that are {\it tight} (see definitions in Section~\ref{section2}), to the graphs $O_k$ and $M_k$
 is undertaken. A definition of $T$ is developed in Section~\ref{s2&half}, while {\it Dyck nests} and {\it tight Dyck nests} are introduced in Section~\ref{section3} associated to corresponding Dyck words and {\it tight} Dyck words via their {\it Dyck path} heights. As said, a {\it castling} procedure assigning tight Dyck nests to the vertices of $T$ is introduced in Section~\ref{sec4}. 

Thus, the main objective of the present work is to apply the tree $T$ (Section~\ref{s2&half}) of tight restricted-growth strings, or TRGSs, (Section~\ref{section2}) lexicographically ordered by the nonnegative integers $n$ to the Dyck words and nests of Section~\ref{section3}, so that to each TRGS $b(n)$ corresponds a specific Dyck nest $f(n)$, and to each child of $T$ corresponds the castling (Section~\ref{sec4}) of its parent node. 
 
 Sections~\ref{section5}-\ref{section6} introduce how to {\it blow} and {\it anchor} these words and nests, adapting them as vertex representatives of the graphs $O_k$ and $M_k$. In Section~\ref{arcf}, an edge-supplementary arc-factorization \cite{DO} of each odd graph based on those representatives is given. 
 
Based on the work of M\"utze et al. \cite{u2f,Hcs}, a permutation is assigned in Section~\ref{section8} to each Dyck nest $F$, leading Section~\ref{1f} to establish uniform 2-factors both in the $O_k$ and the $M_k$, as in \cite{u2f}. 
The said uniform 2-factors yield a partition of $V(O_k)$ (Section~\ref{2f}), and thus a partition of $V(M_k)$, too. However, another partition can be used, instead, based on the action of the cyclic and the dihedral group on $V(O_k)$ and $V(M_k)$, respectively. 

In Section~\ref{sign}, we assign a {\it clone} to each Dyck nest $F$. This clone is equivalent to $F$ (Theorem~\ref{signest}). 
 Moreover, the clone is universally updated along $T$ (Theorem~\ref{id}).
Thus, it generates an infinite sequence of updates; see also \cite{faltaba}. 

A convenient set of strings is established in Section~\ref{conv}, allowing to associate the elements of the graphs $O_k$ and $M_k$,
as arising from the tree $T$ (Theorems~\ref{alfin}-\ref{fidel} and Corollary~\ref{coro}) based on either of the two mentioned partitions.
This allows  
the cited arc-factorization approach to take to the Hamilton cycles of odd graphs and middle-levels graphs of \cite{Hcs} (Section~\ref{Ham}, Theorem~\ref{L6}).

In sum, our approach differs from that of \cite{Hcs} by reversed complementation, so it constitutes an alternative reinterpretation of \cite{Hcs}. That far can be said as for a comparison of this work with that of \cite{Hcs}, as we just have provided a universal  treatment for all odd and middle-levels graph that \cite{Hcs} did not. As our treatment differs from that of \cite{Hcs} just by reversed complementation, adapting \cite{Hcs} to our approach did not offer any additional improvement in time complexity. About the time complexity of castling, see after Example~\ref{catalan}.

\section{Tight restricted-growth strings}\label{section2}

\begin{definition}\label{prec} Given two nonnegative-integer strings $A$ and $B$, of lengths $n>0$ and $\ell\ge n$, respectively, with $B\ne A$ and $B=d_{\ell -1}d_{\ell -2}\cdots d_2d_1$, where $\ell\ge n$, we say that $A$ {\it precedes} $B$ if the concatenated string $0^{\ell -n}|A=c_{\ell -1}c_{\ell -2}\cdots c_2c_1$ has $c_i\le d_i$, for $0<i<\ell $.\end{definition}

\begin{definition} A set of nonnegative-integer strings is said to be {\it lexicographically ordered} if the precedence criterion of Definition~\ref{prec} holds for all its pairs of elements.
\end{definition}

\begin{definition} 
A {\it tight restricted-growth string}, or {\it TRGS} is a nonnegative-integer string of the form
\begin{eqnarray}\label{chris}a_{k-1}a_{k-2}\cdots a_2a_1,\end{eqnarray}
where either $$k>2,\; a_{k-1}=1\mbox{ and }\;a_{j+1}\ge a_j+1,\mbox{ for }k-1>j>0,$$ or $$k=2\mbox{ and }a_1\in\{0,1\}.$$ 
\end{definition}

Let ${\mathcal S}$ be the lexicographically-ordered sequence of TRGS's. Then, ${\mathcal S}$ is the infinite extension of the sequence \cite[\seqnum{A239903}]{oeis}  (finite, as it only uses decimal digit entries $a_j$, while in ${\mathcal S}$ the digit entries $a_j$ grow unbounded). 
 
 \begin{example} ${\mathcal S}$ starts as ${\mathcal S}=(b(0),b(1),b(2),\ldots,b(17),\ldots)=$
\begin{eqnarray}\label{(1)}(0,1,10,11,12,100,101,110,111,112,120,121,122,123,1000,1001,1010,1011,\ldots).\end{eqnarray}
\end{example}

The following definition yields the original notion of restricted-growth string in~\cite[p. 325]{Arndt}.

\begin{definition}\label{def3} A {\it restricted-growth string}, or {\it RGS}, is defined as either a TRGS $b(n)$ or obtained from a TRGS $b(n)$ by prefixing to $b(n)$ a finite string of zeros, e.g. $0^\ell | b(n)$ is an RGS if $0\le\ell\in\mathbb{Z}$ and $b(n)\in{\mathcal S}$.  
\end{definition}

Note that no TRGS starts with 0 if it contains more than one entry. On the other hand, 
01 is an RGS, not an TRGS. 

\section{Tree of tight restricted-growth strings}\label{s2&half}

\begin{definition}\label{mathbbn} Let $\mathbb{N}$ be the set of nonnegative integers.
Let $n\in\mathbb{N}$ and let $\gamma(n)\in[0,k-1]$ be the right-to-left position of the rightmost nonzero entry of $b(n)$ (that is: counting positions from right to left, like the subindices in (\ref{chris})). 
\end{definition} 

\begin{example}
Definition~\ref{mathbbn} yields a sequence $\Gamma=(\gamma(n);0<n\in\mathbb{Z})$ that starts, accompanying ${\mathcal S}\setminus b(0)$, as follows:\\ $\Gamma=(\gamma(1),\ldots,\gamma(17),\ldots)=$ $(1,2,1,1,3,1,2,1,1,2,1,1,1,4,1,2,\ldots)$, or longer, up to $\gamma(43)$:
\begin{eqnarray}\label{(2)}^{\Gamma=(1,2,1,1,3,1,2,1,1,2,1,1,1,4,1,2,1,1,3,1,2,1,1,2,1,1,1,3,1,2,1,1,2,1,1,1,2,1,1,1,1,5,\ldots).}\end{eqnarray}
\end{example}

\begin{definition}\label{treee} For each $0<n\in\mathbb{N}$, let us define a string $c(n)\in{\mathcal S}$ as follows. Express
\begin{eqnarray}\label{yu}\begin{array}{llll}
b(n)=a_{k-1}a_{k-2}\cdots\; a_ja_{j-1}\cdots\!    &\!\cdots a_2a_1,&\mbox{as in (\ref{chris})}, &\mbox {with } j=\gamma(n)\mbox{ as in (\ref{(2)}), and let}\\
c(n)=a_{k-1}a_{k-2}\cdots(a_j-1)a_{j-1}\!&\!\cdots a_2a_1,&&\mbox{unless }j=k-1,\mbox{ in which case}\\
c(n)=b(0)=a_1=0\in\mathcal{S}\mbox{ is the}&\mbox{TRGS}&\mbox{for the}&\mbox{RGS }a_{k\-1}\cdots a_1=0^{k-1}.\end{array}\end{eqnarray}
Let $T$ be the tree obtained by taking $c(n)$ to be the parent of $b(n)$ in $T$, so $T$ is given by
\begin{eqnarray}\label{tree}T=\{{\mathcal S},E(T)\},\mbox{ with }(c(n),b(n))\in E(T),\forall n>0.\end{eqnarray}
\end{definition}

\begin{definition}\label{rho} For each $0<n\in\mathbb{N}$, we have that
$c(n)=b(m)$, for some $m=\rho(n)\in\mathbb{N}$, ($m<n$), where $\rho$ is the {\it  parent} function $\rho:\mathbb{N}\rightarrow \mathbb{N}$ of $T$. This yields the {\it parent sequence} $\rho(1)=0$, $\rho(2)=0$, $\rho(3)=2$, $\rho(4)=3$, $\rho(5)=0$, $\rho(6)=5,$ $\rho(7)=5$, etc., for the edges $(0,1)$, $(0,10)$, $(10,11)$, $(11,12)$, $(0,100)$, $(100,101)$, $(100,110)$, etc.
\end{definition}

\begin{example} Accompanying display (\ref{(2)}), $\rho(\mathbb{N}\setminus\{0\})=\rho(\{1,2,\ldots,43,\ldots\})$ is:
\begin{eqnarray}\label{(3)}_{(0, 0, 2, 3, 0, 5, 5, 7, 8, 7, 10, 11, 12, 0, 14, 16, 17, 14, 19, 19, 20, 21, 19, 24, 25, 26, 14, 28, 28, 30, 31, 30, 33, 34, 35, 33, 37, 38, 39, 40, 0,\ldots).}\end{eqnarray}
Figure~\ref{fig1} further exemplifies Definitions~\ref{treee}--\ref{rho} in the context of Section~\ref{sec4}.
\end{example}

\section{Dyck words, Dyck nests and blowing}\label{section3}

We just have seen that to each $n\in\mathcal{N}$ corresponds a TRGS that constitutes a vertex of the tree $T$, and that $V(T)$ is formed by such TRGSs. In Section~\ref{sec4}, we will see that $V(T)$ is the domain of a {\it castling} correspondence whose image is formed by {\it Dyck nests}, that we introduce in this section.

\begin{definition}\label{def5} By a
 {\it Dyck word} we will understand any binary $2k$-string $f=f_1f_2\cdots f_{2k}$ of weight $k$ with the number of 0-bits at least equal to the number of 1-bits in each prefix of $f$, (which differs from~\cite{gmn,M,u2f,Hcs} just by binary complementation).
 The concept of {\it empty Dyck word} $\epsilon$ also makes sense in this context and is used in Section~\ref{Ham}.
\end{definition}

\begin{definition}\label{dp}  Each Dyck word $f\ne\epsilon$ as in Definition~\ref{def5} determines a {\it Dyck path}, given as a continuous piecewise-linear curve $g_f$ in the Cartesian plane $\mathbb{R}^2$ such that $g_f(0)=g_f(2k)=0$ with $g_f(x)>0$, for $0<x<2k$, and formed by replacing successively from left to right
each 0-bit of $f$ by an {\it up-step} and each 1-bit of $f$ by a {\it down-step}, where {\it up-steps} and {\it down-steps} are segments of the forms $(x,y)(x+1,y+1)$ and $(x,y)(x+1,y-1)$, respectively.
\end{definition}
 
\begin{definition}\label{dq} We assign the integers in the interval $[1,k]$ successively to the up-steps, respectively, down-steps, of $g_f$ in the horizontal unit layers $[y,y+1]\subset\mathbb{R}^2$, for $y=0,1,2,\ldots$, as needed, and {\it from right to left} at each layer. The resulting $2k$-string is said to be the {\it Dyck nest} $F=F_1F_2\cdots F_{2k}$ associated to $f=f_1f_2\cdots f_{2k}$.
\end{definition}

\begin{definition}\label{dr} A Dyck nest $F=\cdots F_{j-1}F_jF_{j+1}F_{j+2}\cdots=\cdots i(i+1)(i+1)i\cdots$, arising from a Dyck word $f=\cdots f_{j-1}f_jf_{j+1}f_{j+2}\cdots=\cdots 0011\cdots$, is said to be obtained by {\it blowing} the shorter Dyck nest $F'=\cdots F'_{j-1}F'_j\cdots=\cdots ii\cdots$, arising from a corresponding Dyck word $f=\cdots f'_{j-1}f'_j\cdots=\cdots 01\cdots$, where $F_{j-1}=F_{j+2}=F'_{j-1}=F'_j=i$, $F_j=F_{j+1}=i+1$, $f_{j-1}=f_j=f'_{j-1}=0$ and $f_{j+1}=f_{j+2}=f'_j=1$.
We also say that $F'$ is obtain by {\it reduction} from $F$, or that the corresponding $f'$ (with $F'=F(f')$) is obtained by {\it reduction} from $f$. In particular, each Dyck word and corresponding Dyck nest are equal or can be reduced to a TRGS and corresponding irreducible Dyck nest. 
\end{definition}

\begin{example} Figure~\ref{fig2} illustrates Definitions~\ref{dp}--\ref{dq}--\ref{dr}. In fact, for each one of the eight cases in the figure, a piecewise-linear curve $g_f$ is constructed iteratively that starts at the shown origin O in the Cartesian plane $\mathbb{R}^2$ by replacing successively the 0-bits and 1-bits of $f$ by {\it up-steps} and {\it down-steps}, namely diagonal segments $(x,y)(x+1,y+1)$ and $(x,y)(x+1,y-1)$, respectively.  We assign the integers in the interval $[0,k]$ in decreasing order (from $k$ to 0) to the up-steps, respectively down-steps, of $g_f$, from the top unit layer of $g_f$ in $\mathbb{R}^2$ to the bottom one and from left to right at each pertaining unit layer between contiguous lines $y,y+1\in\mathbb{Z}$.
Then, by reading and successively writing the number entries assigned to the steps of $g_f$, the $n$-tuple $F(f)$ is obtained. The figure is provided, underneath each instance, with the corresponding $f$ followed by  $F(f)$ and its (underlined) order of presentation via the castling procedure of Section~\ref{sec4}, below. We assume that all elements of $V(O_k)$ are represented by means of such piecewise-linear curves, for each fixed integer $k>0$.

\begin{figure}[htp]
\includegraphics[scale=0.88]{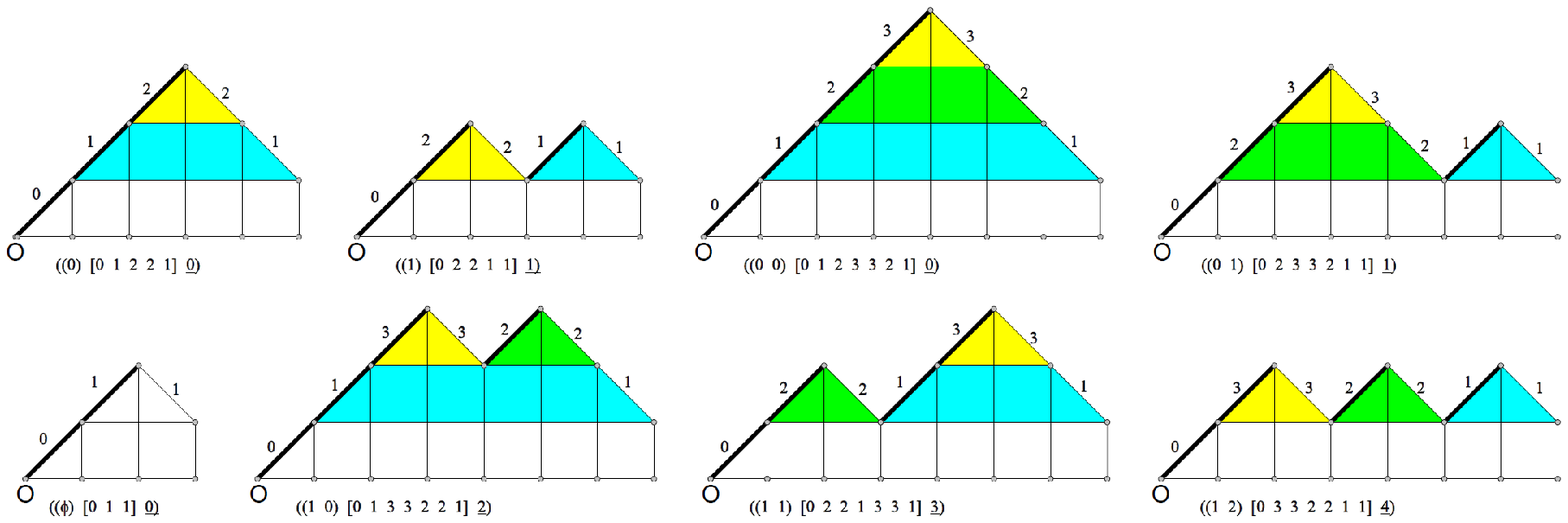}
\caption{Recovering $F(f)$ from $f$: $PLC(f)$ for triples (($f)$ $[F(f)]$, $\underline{ord(f)})$, $k=1,2,3$}
\label{fig2}
\end{figure}

\end{example}

\begin{definition}\label{tight} 
 The inverse operation to blowing, applied initially to a Dyck nest $F$ and iterated successively through $F',(F')',\ldots$ etc,, leads to a {\it tight} Dyck nest, meaning it cannot be further reduced to a shorter Dyck nest. Dyck words $f$ yielding tight Dyck nests $F$ will be said to be {\it tight}.
\end{definition}

Each Dyck word $f$ as in Definition~\ref{def5} determines a Dyck path $g_f$ as in Definition~\ref{dp} which in turn determines a Dyck nest $F(f)$ as in Definition~\ref{dq}.
If $f$ is tight, then $F(f)$ is irreducible, as  in Definition~\ref{dr}, or tight as in Definition~\ref{tight}. If not, then reduction takes $f$ into a tight $f'$, and $F= F(f)$ into an irreducible or tight Dyck nest $F'=F(f')$.

\begin{example} To the Dyck words $f=01,\; f=0011,\; f=000111,\; \ldots,\; f=000\cdots111,$ etc., correspond the respective nests
$F=11,\; F=1221,\; F=123321,\; \ldots,\; F=123\cdots321$, etc.
The first Dyck nest here, $F=11$, is the tight Dyck nest of each of its subsequent ones, and  each of these are obtained by {\it blowing} $F=11$.
\end{example}

\begin{example} Tight Dyck nest $F=2211$ (for tight Dyck word $f=0101$) is the tight Dyck nest of $F=233211$ (for $f=001101$), which is therefore obtained by blowing $F=2211$.
\end{example}

\begin{example} When lexicographically ordered, the tight Dyck nests start as follows, according to (\ref{(1)})-(\ref{(2)}) and (\ref{(3)}):
${\mathcal N}=(F(0),F(1),\ldots,F(13),\ldots)$, where $F(0),F(1),\ldots,F(13),\ldots$ are
\begin{eqnarray}\label{(4)}
_{11,2211,133221, 221331, 332211, 12443321, 24433211, 13324421, 24421331, 33244211, 14433221, 22144331, 33221441, 44332211,\ldots}
\end{eqnarray}
respectively.
The corresponding originating Dyck words here start the sequence of associated {\it tight} Dyck words:
${\mathcal W}=(f(0),f(1),\ldots,f(13),\ldots)$, where $f(0),f(1),\ldots,f(13),\ldots$ are
\begin{eqnarray}\label{(5)}_{01,0101,001011, 010011, 010101, 00010111, 00101101, 00100111, 00110011, 01001101, 00101011, 01001011, 01010011, 01010101,\ldots}\end{eqnarray}
respectively. \end{example}

\section{Castling procedure}\label{sec4}

Modifying Theorem 3.2 \cite{D2} or Theorem 2 \cite{D1} for our present purposes, we assign a Dyck word $f(n)=f_1f_2\cdots f_{2k}$ to each TRGS $b(n)=a_{k-1}\cdots a_1\in{\mathcal S}$  as in (\ref{chris})-(\ref{(1)}) in such a way that all such Dyck words become represented (once each) via uniquely corresponding $F(n)$'s generated via the procedure contained in the subsequent items (a)-(f), each $F(n)$ yielding its $f(n)$ by replacing each first appearance $j_1$ of an integer $j$ by a 0-bit and its second appearance $j_2$ of $j$ by a 1-bit, and using the function $\lambda:\mathbb{N}\rightarrow\mathbb{Z}$ given by $\lambda(n)=k$, ($n\in\mathbb{N}$). 
\begin{enumerate}
\item[(a)] set $n=1$ and set $F(0)=11$, with $\lambda(0)=k=2$.
\item[(b)] Let $\gamma(n)\in\Gamma$ as in (\ref{(2)}), let $(c(n),b(n))\in T$ as in (\ref{tree}), let $0^\ell c(n)=c'(n)=0^\ell b(m)$ as in (\ref{yu})-(\ref{(3)}) ($0\le m<n$) and 
let $j$ be such that $\lambda(q^j(F(m)))=\lambda(F(n))$. 
\item[(c)] Set $q^j(F(m))=W|X|Y|Z$, with $W$ and $Z$ of lengths $\gamma(n)-1$ and $\gamma(n)$, respectively, and with $Y$ starting at entry $x+1$, where $x$ is the leftmost entry of $X$. 
\item[(d)] Set $k=\lambda(n)$ and write $F^k(n)=F(n)$.
\item[(e)] Determine $F^k(n)$ via {\it castling} of $X$ and $Y$, i.e., transposing: $X|Y\rightarrow Y|X$. This yields $F^k(n)=W|Y|X|Z$. 
\item[(f)] Let $n:=n+1$ and repeat the sequence of items (b)-(f).\end{enumerate}

The procedure in items (a)-(f) yields the lexicographically-ordered tight Dyck nests $F(n)=F^k(n)\in{\mathcal N}$ for all $n\ge 0$, and the corresponding Dyck words $f(n)=f^k(n)\in{\mathcal W}$.

\begin{example}\label{catalan} While working with a Dyck word $f$ or a Dyck nest $F$, the order $n\in\mathbb{N}$ of any working tool relating to $f$ or $F$ will be indicated by $ord(\cdot)$.
The RGSs of length $k-1$ conform a tree $T_k$ controlled by the tree $T$ of TRGSs by blowing each such TRGS to an RGS of length $k-1$. Then, $T_k$ can be considered subtree of $T$ corresponding to the prefix $[0,{\mathcal C}_k-1]$ of $\mathbb{N}$, where ${\mathcal C}_k=\frac{(2k)!}{k!(k+1)!}$ is the $k$-th {\it Catalan number} \cite[\seqnum{A000108}]{oeis}, \cite{Stanley}. 
 Fig.~\ref{fig1} contains representations of the effects of the castling operation on the trees $T_k$ ($k=1,2,3,4$, with ${\mathcal C}_k=1,2,5,14$), each such tree with its root $0^{k-1}$ represented in a box containing the order $ord(0^{k-1})=0\in\mathbb{N}$, the root $0^{k-1}$ and $F(0^{k-1})$. 
Each other node $f$ of $T_k$ is represented by a box of two levels: the top level contains the order $ord(\rho(f))\in\mathbb{N}$, the parent $\rho(f)$ and $F(\rho(f))$; the lower level
contains the order $ord(f)\in\mathbb{N}$, $f$ and $F(f)$. 
In these two levels, the entry at which parent $c$ and child $b$ differ in (\ref{yu}), namely at position $\gamma(n)$, is in red in contrast with the remaining entries, in black.   
In all boxes, $F(\rho(f))=``W|X|Y|Z"$ and $F(f)=``W|Y|X|Z"$ have $X$ and $Y$ colored blue and red, respectively, while $W$ and $Z$ are left black. In addition, the edge leading from $\rho(f)$ to $f$ is labeled with its subindex $i$. 

The computational complexity of the castling operation is determined by the value $k=\lambda(n)$ in item (d) in time $O(1)$ and includes: the determination of $\gamma(n)$ and $j$ in item (b), both of time $O(k)$; $W,X,Y,Z$ in item (c), again $O(k)$; the string transposition $XY\rightarrow YX$ of item (e), again $O(k)$; so in sum the time of the castling procedure is $O(k)$, for any particular $n\in\mathbb{N}$, where $k=\lambda(n)$. Considered over all of $T_k$, the complexity amounts to $O(k{\mathcal C}_k)$, where the time complexity of ${\mathcal C}_k$ can be obtained by Dynamic Programming via recursiont ${\mathcal C}_0=1$ and ${\mathcal C}_{k+1}=\sum_{j=0}^{k}{\mathcal C}_j{\mathcal C}_{k-j}$, yielding time $O(k^2)$, a total time $O(k^3)$. Space complexity is $O(k)$.
\end{example}

\begin{figure}[htp]
\includegraphics[scale=0.87]{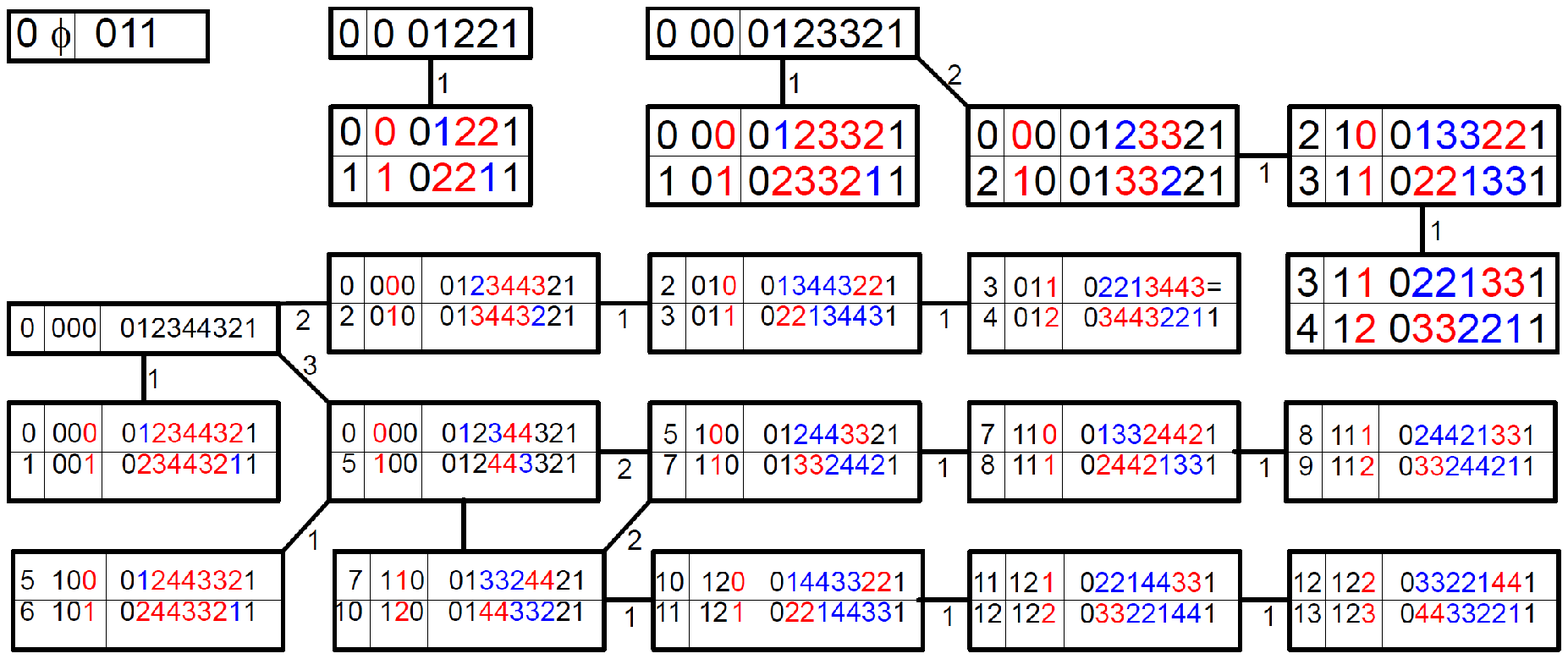}
\caption{A section of $T$ accompanying Definitions~\ref{treee}--\ref{rho} and Section~\ref{sec4}}
\label{fig1}
\end{figure}

\section{Dyck nests blown to a specific length}\label{section5}

\begin{lemma} Every Dyck nest $F=F^k=F(n)=F^k(n)$ of length $2k$ contains a substring $kk=k_1k_2$.\end{lemma}

\begin{proof}
Clearly, the integer $k$ in the statement appears as the substring $kk=k_1k_2$ in $F$. 
\end{proof}

\begin{lemma}\label{z}  Given a Dyck nest $F^k$ of length $2k$, let $F^{k+1}=q(F^k)$ be the string of length $2(k+1)$ obtained from $F^k$ by inserting the substring $(k+1)(k+1)=(k+1)_1(k+1)_2$ between the two entries $k_1$ and $k_2$ that equal $k$ in $F^k$, so that $q(F^k)$ contains the substring $k(k+1)(k+1)k=k_1(k+1)_1(k+1)_2k_2$. Then, $F^{k+1}=q(F^k)$ is a Dyck nest of length $2k+2$ and its corresponding $f^{k+1}=q(f^k)$ is a Dyck word of length $2k+2$.\end{lemma}

\begin{proof}
Clearly, the substring $k(k+1)(k+1)k=k_1(k+1)_1(k+1)_2k_2$ in the larger Dyck nest $F^{k+1}=q(F^k)$, based in the shorter Dyck nest $F^k$ that contains just the substring $kk=k_1k_2$ and is of length $2k$, makes it a Dyck nest of length $2k+2$.
\end{proof}

\begin{definition}\label{wesay} We say that $F^{k+1}=q(F^k)$ and $f^{k+1}=q(f^k)$ are {\it blown} from $F^k$ and $f^k$, respectively, with $q(\cdot)$ as in Lemma~\ref{z}.
By repeating the $q(\cdot)$-operation, we {\it blow} $F^k$ (resp., $f^k$) successively to Dyck nests $q^2(F^k)$, $q^3(F^k)$, etc. (resp., Dyck words $q^2(f^k)$, $q^3(f^k)$, etc.) We also say that $F^k(n)$ and $f^k(n)$ are {\it $k$-blown}, where necessarily $0\le n<{\mathcal C}_k=\frac{(2k)!}{k!(k+1)!}$, with upper bound ${\mathcal C}_k$. 
\end{definition}

\begin{definition}
In a likewise fashion to that of Definition~\ref{wesay}, given an RGS $b(n)\in{\mathcal S}$, we define $q(b(n))=0|b(n)=0b(n)$, assignment that can be iterated as: 
$$q^2(b(n))=00b(n),\ldots,q^i(b(n))=0^ib(n),\mbox{ etc.}$$ \end{definition}

\begin{obs}
$\lambda(n)=k$,  $\forall n\in[{\mathcal C}_{k-1},{\mathcal C}_k-1]$, where $\lambda$ is defined in Section~\ref{sec4}.
\end{obs}

 \begin{lemma} The castling procedure of Section~\ref{sec4} permutes with the $q(\cdot)$-operation, so the compositions $b(n)\rightarrow F(n)\rightarrow q^i(F(n))$ (resp., $b(n)\rightarrow f(n)\rightarrow q^i(f(n))$) and $b(n)\rightarrow q^i(b(n))\rightarrow q^i(F(n))$  (resp., $b(n)\rightarrow q^i(b(n))\rightarrow q^i(f(n))$) conform a commutative diagram.\end{lemma}
 
\begin{proof}
Direct verifications allow to establish the commutativity of the two diagrams:

$$\begin{array}{cccccccc}
b(n)&\rightarrow&F(n)&&& b(n)&\rightarrow&f(n)\\
\downarrow&&\downarrow&&&\downarrow&&\downarrow\\
q^i(b(n))&\rightarrow&q^i(F(n))&&&q^(b(n))&\rightarrow&q^i(f(n))
\end{array}$$

\end{proof}

\begin{obs} {\bf(A)} Recalling Definition~\ref{rho}, $\rho^{-1}(0)=\{{\mathcal C}_k;0<k\in\mathbb{Z}\}$, where ${\mathcal C}_k$ is the smallest $n$ such that $b(n)$ is of length $k$. {\bf(B)} All vertices $b(n)=a_{k-1}\cdots a_1$ of $T$ with $n\notin\rho^{-1}(0)$ have either $\gamma(n)$ or $\gamma(n)-1$ children, depending on whether $a_{\gamma(n)}\le a_{\gamma(n)+1}$ or not, where $\gamma(n)$ is given in Definition~\ref{mathbbn}. \end{obs}

\begin{example}
 {\bf(A)} $b({\mathcal C}_3)=b(5)=100$. {\bf(B)} The children of $b(7)=110$ are $b(8)=111$ and $b(10)=120$, but $b(8)=111$ and $b(10)=120$ have only respective children $b(9)=112$ and $b(11)=121$; the children of $b(30)=1210$ are $b(31)=1211$ and $b(33)=1220$, these having respective children sets $\{b(32)\}=\{1212\}$ and $\{b(34),b(37)\}=\{1221,1230\}$.
\end{example}

\section{Anchoring Dyck words to odd graphs}\label{section6}

\begin{lemma}\label{anch} Prefixing a 0-bit to every $f^k=f^k(n)$ yields exactly ${\mathcal C}_k$ corresponding representatives $0f^k=0f^k(n)$ of the cyclic classes mod $2k+1$ (or $\mathbb{Z}_{2k+1}$-classes) of $V(O_k)$. \end{lemma}

\begin{definition} The representatives in Lemma~\ref{anch} are to be called {\it anchored} $k$-blown Dyck words. The tight versions of such representatives will be said to be {\it tight odd-graph representatives}, or {\it TOGR}'s.\end{definition}

\begin{example} The list of TOGR's corresponding to (\ref{(5)}) starts as follows: 
\begin{eqnarray}\label{-}^{\{001,00101,0001011,0010011,0010101, 000010111, 000101101,000100111,}_{  000110011, 001001101, 000101011, 001001011, 001010011, 001010101\}.}\end{eqnarray}\end{example}

\begin{definition} The tight Dyck nests in (\ref{(4)}) associated to the Dyck words in  (\ref{(5)}) act as the {\it pull-backs} of those Dyck words. That is: each $j_1$ in a Dyck nest $F(n)$ of $\mathcal N$ corresponds to a 0-bit in the corresponding Dyck word $f(n)$ of $\mathcal W$, so $j_2$ in $F(n)$ corresponds to a corresponding 1-bit in $f(n)$. This makes $F(n)$ to be the pull-back of $f(n)$.
\end{definition}

\begin{example}\label{k=3-152} Let $k=3$. Let us begin with the prefix of length ${\mathcal C}_3=5$ of $0{\mathcal W}$ in (\ref{-}):
$$(001, 00101, 0001011, 0010011, 0010101),$$ and transform its TOGR's of length less than $2k+1=7$, namely $0f(0)=0f^1(0)=001$ and $0f(1)=0f^2(1)=00101$, to their pull-back anchored Dyck nests, $0F(0)=0F^1(0)=011$ and $0F(1)=0F^2(1)=02211$, respectively. These two TOGR's are 3-blown to $0F^3(0)=0q^2(011)=0123321$ and $0F^3(1)=0q(02211)=0233211$, that together with the pull-backs of the remaining anchored $3$-blown Dyck words, ($0f^3(2)=0001011$, $0f^3(3)=0010011$ and $0f^3(4)=0010101$), namely the anchored $3$-blown Dyck nests $0F^3(2)=0122331$, $0F^3(3)=0221331$ and $0F^3(4)=0332211$, form the pull-backs of their corresponding anchored $3$-blown Dyck words. Summarizing: $0000111$, $0001101$, $0001011$, $0010011$ and $0010101$; these represent in $V(O_3)$ the ${\mathcal C}_3$ (=5) $\mathbb{Z}_7$-classes $[0f]$, displayed in the upper half of Table~\ref{pull-back}
and correspond to the $\mathbb{Z}_7$-classes $[0F]$ of their pull-back anchored 3-blown Dyck nests $0F$, displayed in the lower half of the table.
\end{example}

\begin{table}[htp]
$$\begin{array}{|clc|}\hline
&[0000111]=\{0000111, 0001110, 0011100, 0111000, 1110000, 1100001,  1000011\}&\\
&[0001101]=\{0001101, 0011010, 0110100, 1101000, 1010001, 0100011, 1000110\}&\\
&[0010110]=\{0001011, 0010110, 0101100, 1011000, 0110001, 1100010, 1000101\}&\\
&[0010011]=\{0010011, 0100110, 1001100, 0011001, 0110010, 1100100, 1001001\}&\\
&[0010101]=\{0010101, 0101010, 1010100, 0101001, 1010010, 0100101, 1001010\}&\\\hline
&[0123321]=\{0123321, 1233210, 2332101, 3321012, 3210123, 2101233, 1012332\}&\\
&[0233211]=\{0233211, 2332110, 3321102, 3211023, 2110233, 1102332, 1023321\}&\\
&[0122331]=\{0122331, 1223310, 2233101, 2331012, 3310122, 3101223, 1012233\}&\\
&[0221331]=\{0221331, 2213310, 2133102, 1331022, 3310221, 3102213, 1022133\}&\\
&[0332211]=\{0332211, 3322110, 3221103, 2211033, 2110332, 1103322, 1033221\}&\\\hline
\end{array}$$
\caption{$\mathbb{Z}_7$-classes of $V(O_3)$ and their pull-back anchored 3-blown Dyck nests}
\label{pull-back}
\end{table}

The characteristic vectors of the vertices of $O_k$ in the upper half of Table~\ref{pull-back} also represent the elements of $L_k\subset V(M_k)$. The complements of their reversed strings represent  $L_{k+1}$ so their union $L_k\cup L_{k+1}$ forms $V(M_K)$, where adjacency is given by inclusion of $ L_k$ into $L_{k+1}$.

\section{Arc factorizations of odd graphs}\label{arcf}

Let\; $^0\!F=0F_1\cdots F_{2k}$ be an anchored Dyck nest representing a $\mathbb{Z}_{2k+1}$-class of  $V(O_k)$ ($F_i\in[1,k]$, for $i\in[1,2k]$). Then, such a class can be expressed (Example~\ref{k=3-152}, Table~\ref{pull-back}) as $[^0\!F]=$
\begin{eqnarray}\label{^0F}\{^0\!F\!=\!0F_1\cdots F_{2k},\;\; ^1\!\!F\!=\!F_1\cdots F_{2k}0,\;\; ^2\!F\!=\! F_2\cdots F_{2k}0F_1,\; \ldots,\;\; ^{2k}\!F\!=\!F_{2k}0F_1\cdots F_{2k-1}\}.\end{eqnarray}
Each entry of value $j\in[0,k]$ in a $(2k+1)$-tuple $\chi=\,^\ell\!F$ as in (\ref{^0F}) ($\ell\in[0,2k]$) is either the {\it first}, $j_1$, or {\it second}, $j_2$, {\it appearance} of $j$ in $\chi$ counting from the entry of value 0, from left to right and cyclically mod $2k+1$, e.g., $^2\!F$ has value 0 in the penultimate entry  (see (\ref{^0F})).

Let  $\ell\in[0,2k]$.
Given the entry of value $j=0$ in $^\ell\!F$ or the first appearance $j_1$ of an integer $j\in[1,k]$ in $^\ell\!F$, there exists just one anchored $k$-blown Dyck nest 
$_{j+\ell}^{\hspace*{4mm}\ell}\!F\ne\,^\ell\!F$ (with $j+\ell$ taken mod $n$) adjacent as a vertex of $O_k$ to $^\ell\!F$ via one of the two arcs forming an edge $e$ of $O_k$,
namely the arc $\overrightarrow{e}=(^\ell\!F,$ $_{j+\ell}^{\hspace*{4mm}\ell}\!F)$,
and having:
\begin{enumerate}\item $(k-j)_1$ as first appearance of $k-j$ in the same position as $j_1$ in $^0\!F$, and
\item the other $j=0$ or first (resp., second) appearance positions of integers $j\in[1,k]$ in $^\ell\!F$ as second  (resp., $j=0$ or first) appearance positions of integers $j\in[1,k]$ in $_{j+\ell}^{\hspace*{4mm}\ell}\!F$.\end{enumerate}

\begin{table}[htp]
$$\begin{array}{|ccccccc|}\hline
0^0_3=|0|123321 && 0^0_3=0|1|23321 && 0^0_3=01|2|3321 && 0^0_3=012|3|321 \\
0^4_3=|3|321012 && 2^4_3=3|2|21013 && 1^4_3=32|1|1023 && 0^3_3=321|0|123 \\
&&&&&&\\
1^0_3=|0|233211 && 1^0_3=0|2|33211 && 1^0_3=02|3|3211 && 1^0_3=02332|1|1 \\
3^3_3=|3|310221 && 4^3_3=2|1|10332 && 2^2_3=21|0|1332 && 0^3_3=32101|2|3 \\
&&&&&&\\
2^0_3=|0|133221 && 2^0_3=0|1|33221 && 2^0_3=01|3|3221 && 2^0_3=0133|2|21 \\
1^0_3=|3|321102 && 4^5_3=3|2|21103 && 3^2_3=21|0|3312 &&  0^3_3=3210|1|23 \\
&&&&&&\\
3^0_3=|0|221331 && 3^0_3=0|2|21331 && 3^0_3=022|1|331 && 3^0_3=0221|3|31 \\
2^5_3=|3|322101 && 3^5_3=2|1|33102 && 3^2_3=310|2|213 && 1^4_3=3211|0|23 \\
&&&&&&\\
4^0_3=|0|332211 && 4^0_3=0|3|32211 && 4^0_3=033|2|211 && 4^0_3=03322|1|1 \\
4^6_3=|3|322110 && 4^1_3=1|0|33221 && 2^2_3=210|1|332 && x^0_3=32110|2|3\\\hline
\end{array}$$
\caption{Twenty adjacencies in $O_3$}
\label{tabla1}
\end{table}

\begin{remark}\label{assign}Assigning color $j$ to the arc $\overrightarrow{e}$ implies assigning color $(k-j)$ to its opposite arc $\overleftarrow{e}=$  $(_{j+\ell}^{\hspace*{4mm}\ell}\!F,$ $^\ell\!F)$. This insures an arc factorization of $O_k$ that we call the {\it edge-supplementary 1-arc-factorization of} $O_k$, since by considering the color set $[0,k]$ as if it were a set of weights with values in $\mathbb{N}$, then the weight of each edge seen as the sum of the weights of its two arcs is constantly equal to $k$ in $O_k$. This justifies our supplementary terminology.\end{remark}

 We use the following notation for the vertices $v\in V(O_k)$: $v=n_k^j$, where $n$ is the order of the TRGS $b(n)\in{\mathcal S}$ yielding the $\mathbb{Z}_{2k+1}$-class of $v$, and $j$ is taken by expressing $v$ as $^j\!F$ in (\ref{^0F}).

Thus, for each $k\in\mathbb{N}$, the arcs of $O_k$ form an {\it arc factorization} $\mathcal F$ composed by the {\it arc factors} $\mathcal F_j$ formed by those arcs $(n^j_k,\ell^{j'}_k)$ with fixed $j\in[0,k]$ and adequate $n$, $\ell$ and $j'$.

\begin{example} Table~\ref{tabla1} shows the twenty adjacency pairs that represent the arcs $(n^0_k,\ell^j_k)$, for $k=3$, in a sandwich fashion, with each of the two shown layers being ``\,$v^j_k=\,^j\!F$\,"; $^j\!F$ shown with concatenation bars delimiting
either $j=0$ or the first appearance $j_1$ of $j\in[1,2k]$ on top (showing $n_k^0$) and the first appearance
$(k-j)_1$ of $k-j$ below (showing $\ell_k^j$).
\end{example}

Lifting $\mathcal F$ via pull-back from $O_k$ to $M_k$ via the 2-covering graph map $\Psi_k:M_k\rightarrow O_k$ yields the so-called {\it modular} arc factorization of $M_k$ \cite{DKS,u2f,Hcs}, again mentioned in relation to item 2 in Section~\ref{2f}.

\section{Permutations associated to Dyck words}\label{section8}

Let us associate a $2k$-permutation $\pi$ to each anchored $k$-blown Dyck nest
$F$, or word $f$, (and to the $\mathbb{Z}_{2k+1}$-class of $V(O_k)$ they represent), via the following procedure (compare \cite{DO,Hcs}):

\begin{enumerate}\item
Set parentheses or commas between each two entries of $f$, so that the four substrings
$$\begin{array}{rrrrrrrrl}
 ``01"&,&``10"&,&``00"&\mbox{and}&``11"&&\mbox{ are transformed into the substrings}\\
 ``0,1"&,&``1)(0"&,&``0(0"&\mbox{and}&``1)1"&,&\mbox{ respectively, resulting in a string }f'.
 \end{array}$$
Add a terminal parenthesis to $f'$, so that the last ``$1$" in $f'$ is transformed into ``$1)$". Denote by $g$ the string resulting from such addition of a closing parenthesis to $f'$.
\item
By proceeding from left to right, replace the bits of $g$ by the successive integers from 1 to $|g|$, keeping all pre-inserted parentheses and commas of $g$ in their position. This yields a version $h_0$ of $g$.
 \item
 Set $h_0$ as a concatenation $(w_1)|(w_2)|\cdots|(w_t)$ of expressions $(w_i)$, ($1\le i\le t$), the terminal ``)" of each $(w_i)$  being the closing ``)" nearest to its opening  ``(".  Let $w'_i$ be the number string obtained from $w_i$ by removing parentheses and commas. For $i=1,\ldots,t$,
perform a recursive step ${\mathcal R}$ consisting in transforming $w'_i$ into its reverse substring $w''_i$ and then resetting $w''_i$ in place of $w'_i$ in $(w_i)$, with the parentheses and commas of $(w_i)$ kept in place. Denote the resulting expression by ${\mathcal R}(w_i)$.
This yields a string $h_1={\mathcal R}(w_1)|{\mathcal R}(w_2)|\cdots|{\mathcal R}(w_t).$

\begin{table}[htp]
$$\begin{array}{|c|c|c|c|c|}\hline
b(0)\rightarrow 00&b(1)\rightarrow 01&b(2)\rightarrow 10&b(3)\rightarrow 11&b(4)\rightarrow 12\\\hline
(1(2(3,3)2)1) & (2(3,3)2)(1,1) & (1(3,3)(2,2)1) & (3,3)(1(2,2)1) & (3,3)(2,2)(1,1)\\\hline
(0(0(0,1)1)1) & (0(0,1)1)(0,1) & (0(0,1)(0,1)1) &  (0,1)(0(0,1)1) & (0,1)(0,1)(0,1)  \\\hline
(1(2(3,4)5)6) & (1(2,3)4)(5,6) & (1(2,3)(4,5)6) &  (1,2)(3(4,5)6) & (1,2)(3,4)(5,6)  \\
(6(5(4,3)2)1) & (4(3,2)1)(6,5) & (6(5,4)(3,2)1) &  (2,1)(6(5,4)3) & (2,1)(4,3)(6,5)  \\
(6(2(3,4)5)1) & (4(2,3)1)(6,5) & (6(4,5)(2,3)1)  &  (2,1)(6(4,5)3) &  \\
 (6(2(4,3)5)1) &                       &                        &                        &\\\hline
p=624351 & p=423165 & p=645231 &  p=216453 & p=214365  \\
\iota=123456&\iota=123456&\iota=123456&\iota=123456&\iota=123456\\
\pi=624351 & \pi=423165 & \pi=645231 & \pi=216453 & \pi=214365\\\hline
\end{array}$$
$$\begin{array}{|c|c|c|c|}\hline
b(1)\rightarrow 001&b(2)\rightarrow 010&b(7)\rightarrow 110&b(10)\rightarrow 120\\\hline
(2(3(4,4)3)2)(1,1)&(1(3(4,4)3)(2,2)1)&(1(3,3)(2(4,4)2)1)&(1(4,4)(3,3)(2,2)1)\\\hline
(0(0(0,1)1)1)(0,1)&(0(0(0,1)1)(0,1)1)&(0(0,1)(0(01)1)1) &(0(0,1)(0,1)(0,1)1)\\\hline
(1(2(3,4)5)6)(7,8)&(1(2(3,4)5)(6,7)8)&(1(2,3)(4(5,6)7)8)&(1(2,3)(4,5)(6,7)8)\\
(6(5(4,3)2)1)(8,7)&(8(7(6,5)4)(3,2)1)&(8(7,6)(5(4,3)2)1)&(8(7,6)(5,4)(3,2)1)\\
(6(2(3,4)5)1)(8,7)&(8(4(5,6)7)(2,3)1)&(8(6,7)(2(3,4)5)1)&(8(6,7)(4,5)(2,3)1)\\
(6(2(4,3)5)1)(8,7)&(8(4(6,5)7)(2,3)1)&(8(6,7)(2(4,3)5)1)&\\\hline
       p=62435187&        p=84657231& p=86724351&p   =86745231\\
\iota=12345678&\iota=12345678 &\iota=12345678&\iota=12345678\\
   \pi=62435187&\pi=86724351& \pi      =84657231&\pi   =86745231\\\hline
\end{array}$$
\caption{Examples of determination of the permutation $\pi$}
\label{exapp}
\end{table}

\item
For $i\in[1,t]$, let ${\mathcal R}(w_i)=(a_{i,1}^1\eta_{i,1}^1b_{i,1}^1)|(a_{i,2}^1\eta_{i,2}^1b_{i,2}^1)|\cdots|(a_{i,t_i}^1\eta_{i,t_i}^1b_{i,t_i}^1)$, where $a_{i,j},b_{i,j}\in\mathbb{N}$ and $\eta_{i,j}^1=(w_{i,j})$ has terminal ``)" being the closing ``)" nearest to its opening ``(". Apply the treatment of the $(w_i)$'s  in item 3 to each $\eta_{i,j}=(w_{i,j})\ne$``(,)", for $j\in[1,t_i]$. Replace the resulting strings ${\mathcal R}(w_{i,j})$ in place of the corresponding $(w_{i,j})$ in ${\mathcal R}(w_i)$, yielding a modified version ${\mathcal R}^2(w_i)$ of ${\mathcal R}(w_i)$. Let $h_2={\mathcal R}^2(w_1)|{\mathcal R}^2(w_2)|\cdots|{\mathcal R}^2(w_t)$.
\item
Each ${\mathcal R}(w_{i,j})$ is a concatenation of terms of the form $a^2_I\eta^2_Ib^2_I$ with $I=\{i,j_1,j_2\}$, where $j_1:=j$. In each such concatenation, the strings $\eta^2_I\ne$``(,)" are of the form $(w_I)$ and must be treated as $(w_{i,j})$ is in item 4 (or $(w_i)$ in item 3), producing a modified string ${\mathcal R}(w_I)$ that forms part of the subsequent string $h_3$. Eventually ahead, to pass from $h_{\ell-1}$ to $h_\ell$ ($\ell>3$), each ${\mathcal R}(w_I)$ in $h_{\ell-1}$ with $I=\{i,j_1,\ldots,j_{\ell-2}\}$ would be a concatenation of terms of the form $a^{\ell-1}_{I'}|\eta^{\ell-1}_{I'}|b^{\ell-1}_{I'}$ with $I'=\{i,j_1,\ldots,j_{\ell-1}\}$. In each such concatenation, those $\eta^{\ell-1}_{I'}\ne$``(,)" would be of the form $(w_{I'})$, to be treated again as in items 3-4.
\item
A sequence $(h_0,\ldots,h_{s+1})$ is eventually obtained for some $s\ge 0$ when all innermost expressions $(w_I)=(a,a\pm 1)$ with $a,a\pm 1\in[1,2k]$ have been processed.
Disregarding parentheses and commas in $h_{s+1}$ yields a $2k$-string $g'$ and an assignment $i\rightarrow p(i)$, ($i\in[1,2k]$), by making correspond the places $i\in[1,2k]$ of $g'$ to the values in the places of $g'$.
Define $\pi=p^{-1}$, the inverse $2k$-permutation of $p=(p(1)p(2)\cdots p(2k))$.
\end{enumerate}

\begin{example} Table~\ref{exapp} illustrates the determination of the permutation $\pi$ for the five cases with $k=3$ and just four of the 14 cases with $k=4$, (exemplifying  that not necessarily $\pi=p$). Each such case is headed by an indication $b(n)\rightarrow b'(n)=0^{k-\ell(n)}|b(n)$, where $\ell(n)$ is the length of $b(n)$ and $k$ is the length of $b'(n)$. The second and third lines of Table~\ref{exapp} show $F'(b'(n))$ and $f'(b'(n))$ with parentheses and commas as in item 1 of the procedure. The rest of each case follows items 2-6 in order to produce $p$ expressed without parentheses or commas, followed by the identity permutation $\iota=12\cdots (2k)$ to ease visualizing $\pi$ as the inverse permutation of $p$, in the final line of each case of the table.
\end{example}

\section{Uniform 2-factors of the odd graphs}\label{1f}

\begin{table}[htp]
$$\begin{array}{|l||l|l||l|l|l|l|l|}\hline
^{k=1}_{n=0}\!&^{k=2}_{n=0}\!&^{k=2}_{n=1}\!&^{k=3}_{n=0}\!&^{k=3}_{n=1}\!&^{k=3}_{n=2}\!&^{k=3}_{n=3}\!&^{k=3}_{n=4}\\
001\!&00011\!&00101\!&0\ol{00011{\ul1}}|\xi_{1^\epsilon}^3 \!&0\ol{0{\ul0}1101}|{\ul{\xi}}_3^2\!&0\ol{00101{\ul1}}|\xi_{1^\epsilon}^3\!&0\ol{0100{\ul1}1}|\xi_3^2\!&0\ol{{\ul0}10101}|\xi_4^3\\
          \!&          \!&                                          \!&0\ol{0001{\ul1}1}|\xi_{1^\epsilon}^1                              \!  &                     \! &                  \! &           \!&\\\hline
0\ul{1}1|_1\!&0\ul{1}221|_1\!&022\ul{1}1|_{3}\!&0\ul{1}233{2}{1}|_{100}\!&02{3}32\ul{1}1|_{510}\!&0\ul{1}3322{1}|_{120}\!&022\ul{1}3{3}1|_{330}\!&0{3}322\ul{1}1|_{540}\\
10\ul{1}|_2\!&211\ul{0}1|_3\!&2101\ul{2}|_{4}\!&32210\ul{1}3|_{524}\!&321012\ul{3}|_{603}\!&322\ul{1}103|_{345}\!&31022\ul{1}3|_{532}\!&321102\ul{3}|_{614}\\
\ul{1}10|_0\!&11\ul{0}22|_2\!&1\ul{2}210|_{1}\!&1\ol{10\ul{3}322}|_{342}^{0\downarrow2}\!&1\ul{2}33210|_{106}\!&11\ul{0}2332|_{212}\!&1\ol{331\ul{0}22}|_{434}^{1\downarrow0}\!&133\ul{2}210|_{326}\\
          \!&1022\ul{1}|_4\!&10\ul{1}22|_{2}\!&2110\ul{3}32|_{443}\!&211\ul{0}233|_{313}\!&21331\ul{0}2|_{535}\!&210133\ul{2}|_{622}\!&2101\ul{2}33|_{402}\\
          \!&\ul{2}1102|_0\!&\ul{2}2110|_{0}\!&22\ul{1}1033|_{244}\!&22\ul{1}3310|_{236}\!&2210\ul{1}33|_{423}\!&2\ul{3}32101|_{105}\!&2\ul{3}32110|_{116}\\
          \!&                     \!&                     \!&102332\ul{1}|_{601}\!&1\ol{022\ul{1}33}|_{431}^{0\uparrow0}\!&103322\ul{1}|_{641}\!&10\ul{1}2332|_{201}\!&1\ol{0\ul{1}3322}|_{221}^{1\uparrow3}\\
          \!&                     \!&                     \!&\ul{3}\ol{321012}|_{{004}}^{1\downarrow4}\!&\ul{3}310221|_{033}\!&\ul{3}\ol{321102}|_{015}^{0\downarrow1}\!&\ul{3}322101\:_{025}&\ul{3}322110|_{006}\\\hline
\end{array}$$
\caption{Uniform 2-factors provided by permutations $\pi$ for $k=1,2,3$.}
\label{uniform}
\end{table}

Departing from each anchored $k$-blown Dyck word taken as a vertex $v$ of $O_k$, an oriented path $\overrightarrow{P^{2k}_{\!v}}$ of length $2k$ is grown in $2k$ stages $v=v_0$, $\overrightarrow{P_v^1}=v_0v_1$, $\overrightarrow{P_v^2}=v_0v_1v_2$, $\ldots$, $\overrightarrow{P^{2k}_{\!v}}=v_0v_1\cdots v_{2k}$ by traversing successively from $v_i$ ($i\in[0,2k-1]$) the edge arc whose assigned color, according to Remark~\ref{assign}, is the corresponding entry of the reversed permutation $rev(\pi)$ of $\pi$. Now, the terminal vertex $v_{2k}$ of $\overrightarrow{P^{2k}_{\!v}}$ is at distance 1 from $v$ by means of an arc $\overrightarrow{e_v}$ in the arc factor $\mathcal F_0$. An oriented $(2k+1)$-cycle $\overrightarrow{C_v^k}$ in $O_k$ is created by adding $\overrightarrow{e_v}$ to $\overrightarrow{P^{2k}_{\!v}}$. Lifting such $\overrightarrow{C_v^k}$ to $M_k$ via $\Psi_k^{-1}$ yields an oriented $2(2k+1)$-cycle $\overrightarrow{_MC_v^k}$ in $M_k$ containing all pairs of opposite vertices $(w,\aleph(w)$) (i.e, at distance $k$ from each other along $\overrightarrow{_MC_v^k}$ via two internally disjoint  paths, one oriented, the other one anti-oriented). This provides $O_k$ (resp., $M_k$) with a 2-factor of ${\mathcal C}_k$ components, all as
oriented cycles of uniform length $2k+1$ (resp., $2(2k+1)$) \cite{DO,u2f,Hcs}.

\begin{table}[htp]
$$\begin{array}{||r|r|r|r||r|r|r|r||r|r|r|r||}\hline\hline
i&n_0^i&s^i&\gamma^i&i&n_0^i&s^i&\gamma^i&i&n_0^i&s^i&\gamma^i\\\hline\hline
0& 0&1&*&14&42&1&5&28&84&1&4\\
1& 2&2&2&15&44&2&2&29&86&2&2\\
2&5&1&3&16&47&1&3&30&89&1&3\\
3& 7&2&2&17&49&2&2&31&91&2&2\\
4&10&3&2&18&52&3&3&32&94&3&2\\
5&14&1&4&19&56&1&4&33&98&1&3\\
6&16&2&2&20&58&2&2&34&100&2&2\\
7&19&1&3&21&61&1&3&35&103&3&2\\
8&21&2&2&22&63&2&2&36&107&4&2\\
9&24&3&2&23&66&3&2&37&112&1&3\\
10&28&1&3&24&70&1&3&38&114&2&2\\
11&30&2&2&25&72&2&2&39&117&3&2\\
12&33&3&2&26&75&3&2&40&121&4&2\\
13&37&4&2&27&79&4&2&41&126&5&2\\\hline\hline
\end{array}
\qquad
\begin{array}{||r|r|r|r|r||r||}\hline\hline
\;\;\;3&\;\;\;2&\;\;\;2&\;\;\;2&\;\;\;2&\;\;\;2\\\hline\hline
0&-3&-4&-5&-6&0\\
0&1&2&3&4&-2\\
&0&1&2&3&-3\\
&&0&1&2&-4\\
&&&0&1&-5\\
&&&&0&-6\\
&&&&&0\\\hline\hline
\end{array}
$$
\caption{Table of initial values of $n_0$, $s$ and $\gamma(n_0)$. Final data of Table~\ref{of}}
\label{quilombo}
\end{table}

\begin{example} Table~\ref{uniform} is headed by  $k=1,2,3$ and $n\in[0,{\mathcal C}_k-1]$ followed by the corresponding anchored Dyck words $0f^k(n)$; for $k=3$ it contains information clarified in Section~\ref{Ham}. 
Below its second horizontal line, the table presents each $\overrightarrow{C_v^k}+=\overrightarrow{P_v^k}+\overrightarrow{e_v^k}
$ in vertical fashion, with each pair of contiguous downward lines, say $\chi,\chi'$, representing an arc whose $j$-th entry is underlined, where $j$, shown as a subindex to the right, is the position containing the
pair of $k$-supplementary entries in $\chi,\chi'$. The vertical column of such subindices $j$ conform the reversed permutation $rev(\pi_v)$ of $\pi_v$ associated to the anchored $k$-blown Dyck word $v$, for each such word $v$ in $V(O_k)$. For each value of $k$, the columns of $(2k+1)$-tuples $\chi$ will be denoted $L(n)$ ($0\le n<{\mathcal C}_k$). For $k=3$, the subindex $j$ of each $(2k+1)$-tuple $\chi$ in a vertical list $L(n)$ is further extended, first with the value $m$ of the TRGS $b(m)$ such that $0F^k(m)$ is the Dyck nest representing the corresponding $\mathbb{Z}_{2k+1}$-class $[0F^k(m)]$ of $\chi$ in $V(O_k)$, and second with the index $j'$ such that $\chi=\,^{j'}\!F^k(m)$, in the notation of (\ref{^0F}). Additional underlined entries and superindices with middle up-or-down vertical arrows are explained in Section~\ref{Ham}.
\end {example}

\section{Partitions of odd-graph vertex sets}\label{2f}

 So far, we have two different partitions of $V(O_k)$ into $(2k+1)$-subsets. In terms of anchored $k$-blown Dyck nests, these partitions are:
\begin{enumerate}
\item the $\mathbb{Z}_{2k+1}$-classes $[0F^k(n)]$ of $V(O_k)$, for $0\le n<{\mathcal C}_k$ in Sections~\ref{section6}-\ref{section8};
\item the vertex sets $V\left(\overrightarrow{C_v^k}\right)$, where $v=0F^k(n)$, for $0\le n<{\mathcal C}_k$ in Section~\ref{1f}, from \cite{u2f}.
\end{enumerate}

Item 1 refers to the {\it algebraic partition} of $V(O_k)$ that leads to the determination of Hamilton cycles in $M_k$ via the {\it lexical} 1-factorization of $M_k$  \cite{D2,D1,gmn,M}. Here, the $\mathbb{Z}_{2k+1}$-classes of $O_k$ correspond via the inverse image $\Psi_k^{-1}$ to the dihedral classes, or ${\mathbb D}_{2k+1}$-classes, of $M_k$, where ${\mathbb D}_{2k+1}$ is the dihedral group with $2(2k+1)$ elements in which the cyclic group ${\mathbb Z}_{2k+1}$ appears as a subgroup of index 2.

\begin{table}[htp]
$$\begin{array}{||r|r|r|r|r||r|r|r|r|r||r|r|r|r|r||}\hline\hline
n^i_j&b()&\rho()&\gamma()&h()&n^i_j&b()&\rho()&\gamma()&h()&n^i_j&b()&\rho()&\gamma()&h()\\\hline\hline
0^0_0&0&-&-&-&14^5_0&1000&0&4&0&28^{10}_0&1200&19&3&0\\
1^0_1&  1&0&1&0&15^5_1&1001&14&1&0&29^{10}_1&1201&28&1&0\\
2^1_0&10&0&2&0&16^6_0&1010&14&2&0&30^{11}_0&1210&28&2&-3\\
3^1_1&11&2&1&-2&17^6_1&1011&16&1&-2&31^{11}_1&1211&30&1&1\\
4^1_2&12&3&1&0&18^6_2&1012&17&1&0&32^{11}_2&1212&31&1&0\\
5^2_0&100&0&3&0&19^7_0&1100&14&3&-4&33^{12}_0&1220&30&2&-4\\
6^2_1&101&5&1&0&20^7_1&1101&19&1&0&34^{12}_1&1221&33&1&2\\
7^3_0&110&5&2&-3&21^8_0&1110&19&2&1&35^{12}_2&1222&34&1&2\\
8^3_1&111&7&1&1&22^8_1&1111&21&1&-3&36^{12}_3&1223&35&1&0\\
9^3_2&112&8&1&0&23^8_2&1112&22&1&0&37^{13}_0&1230&33&2&0\\
10^4_0&120&7&2&0&24^9_0&1120&21&2&0&38^{13}_1&1231&37&1&-2\\
11^4_1&121&10&1&-2&25^9_1&1121&24&1&-2&39^{13}_2&1232&38&1&-3\\
12^4_2&122&11&1&-3&26^9_2&1122&25&1&1&40^{13}_3&1233&39&1&-4\\
13^4_3&123&12&1&0&27^9_3&1123&26&1&0&41^{13}_4&1234&40&1&0\\\hline\hline
\end{array}$$
\caption{The first 42 TRGS's $b(n^i_j)$ of ${\mathcal S}$ with associated data $n^i_j$, $\rho(n^i_j)$, $\gamma(n^i_j)$ and $h(n^i_j)$.}
\label{dia23}
\end{table}

In fact, each anchored $k$-blown Dyck word $0f^k(n)$ is a binary $(2k+1)$-string of weight $k$ whose support is a vertex of $O_k$ as well as an element of $L_k$, while $\aleph(0f(n))$ is an element of $L_{k+1}$. Note that both the pair $\{0f(n),\aleph(0f(n))\}$ and the ${\mathbb Z}_{2k+1}$-class of $0f(n)$ in $L_k$ ($=V(O_k)$) generate together the $\mathbb{D}_{2k+1}$-class of $0f(n)$ of $V(M_k)$. Thus, $0f(n)$ represents both a ${\mathbb Z}_{2k+1}$-class of $V(O_k)$ and a ${\mathbb D}_{2k+1}$-class of $V(M_k)$.

Item 2 refers to the {\it graph-theoretical} partition of $V(O_k)$ that leads to the determination of Hamilton cycles both in $O_k$ and $M_k$ ($k>3$) \cite{Hcs}
via the {\it modular} arc factorization of $M_k$ mentioned at the end of Section~\ref{arcf}. 
Note that the Petersen graph $M_3$ is hypo-hamiltonian, a constraint for Section~\ref{Ham} and its Theorem~\ref{L6} that reformulate those determinations.

Let $(P^0,P^1,P^2,\ldots)$ be a partition of $V(T)$ into {\it threads} $P^i$, each thread inducing a path $T[P^i]$
with an initial vertex $b(n_0^i)$ such that $\gamma(n_0^i)>1$ and its remaining vertices $b(n_j^i)=b(n_0^i+j)$ such that $\gamma(n_j^i)=\gamma(n_0^i+j)=1$, for $0<j\le s^i$,
where the length $s^i$ of $P^i$ is maximal. Here, the indices $n_0^i$ form an integer sequence $(n_0^0,n_0^1,n_0^2,\ldots)$,
as shown vertically on the second, sixth and tenth columns on the triptych left of Table~\ref{quilombo}.
The induced paths $T[P^i]$ have respective lengths $s^i$ and values $\gamma^i=\gamma(n_0^i)$ shown subsequently in the triptych.

Table~\ref{dia23} shows the first 42 TRGS $b(n)=b(n^i_j)$, with the columns (divided again as in a triptych) headed $n^i_j$, $b(n^i_j)$, $\rho(n^i_j)$, $\gamma(n^i_j)$ and $h(n^i_j)$,
where $h(.)$ is introduced in Observation~\ref{obs}.

We reunite the threads $P^i$ into {\it braids}, which are subsets $Q^\ell$ ($0\le \ell\in\mathbb{N}$) of $V(T)$, each inducing a maximal connected ordered subtree
with initial vertex $b(m_0^\ell)$ such that $\gamma(m_0^\ell)>2$ and its remaining vertices $b(m_j^\ell)$ such that $\gamma(m_j^\ell)=\gamma(m_0^\ell+j)\in\{1,2\}$, for $0<j\le \sum_{P^i\subseteq Q^\ell}s^i$. This yields a partition $\{Q^\ell\}$ of $V(T)$ coarser than $\{P^i\}$.

\section{Clones of Dyck nests}\label{sign}

\begin{table}[htp]
$$\begin{array}{||c||r|r||r|r||c||r|r||r|r||r||}\hline\hline
k&\rho(n)&n&b(\rho(n))&b(n)&\gamma(n)&0F(\rho(n))&0F(n)&\sigma(\rho(n))&\sigma(n)&\sigma_{\gamma(n)}(n)\\\hline\hline
1&*&0&*&0&*&0&011&*&1&\\\hline
2&0&1&0&1&1&0{\ul 1}{\it 22}1&0{\it 22}{\ul 1}1&1&0&0\\\hline
3&0&2&00&10&2&01{\ul 2}{\it 33}21&01{\it 33}{\ul 2}21&12&02&0\\
3&2&3&10&11&1&0{\ul{133}}{\it 22}1&0{\it 22}{\ul{133}}1&02&01&k-2\\
3&3&4&11&12&1&0{\ul{221}}{\it 33}1&0{\it 33}{\ul{221}}1&01&00&0\\\hline
4&0&5&000&100&3&012{\ul 3}{\it 44}321&012{\it 44}{\ul 3}321&123&023&0\\
4&5&6&100&101&1&0{\ul 1}{\it 244332}1&0{\it 244332}{\ul 1}1&023&020&0\\
4&5&7&100&110&2&01{\ul{244}}{\it 33}21&01{\it 33}{\ul{244}}21&023&013&k-3\\
4&7&8&110&111&1&0{\ul{133}}{\it 2442}1&0{\it 2442}{\ul{133}}1&013&011&1\\
4&8&9&111&112&1&0{\ul{24421}}{\it 33}1&0{\it 33}{\ul{24421}}1&011&010&0\\
4&7&10&110&120&2&01{\ul{332}}{\it 44}21&01{\it 44}{\ul{332}}21&013&003&0\\
4&10&11&120&121&1&0{\ul{14433}}{\it 22}1&0{\it 22}{\ul{14433}}1&003&002&k-2\\
4&11&12&121&122&1&0{\ul{22144}}{\it 33}1&0{\it 33}{\ul{22144}}1&002&001&k-3\\
4&12&13&122&123&1&0{\ul{33221}}{\it 44}1&0{\it 44}{\ul{33221}}1&001&000&0\\\hline\hline
\end{array}$$
\caption{Development of clone updates $\sigma_{\gamma(n)}(n)$ expressed in terms of $k$}
\label{signa}
\end{table}

Each anchored $k$-blown Dyck nest $0F(n)$ (arising from its anchored $k$-blown Dyck word $0f(n)$) is encoded by its {\it clone} or {\it signarture}, defined as
the vector $\sigma(n)=(\sigma_{k-1}(n),\ldots,\sigma_2(n),$\\ $\sigma_1(n))$  of halfway-distance floors $\sigma_j(n)$ between the first, $j_1$, and second, $j_2$, appearances of each integer $j$ assigned to the respective up- and down-steps of the path $g_f$, where $k>j>0$.

For example, if $j_1k_1k_2j_2$, (or $j_1(k-1)_1k_1k_2(k-1)_2j_2$), is a substring of $0F(n)$, (resp., $0F(n')$), then the halfway-distance floor of $j$ is $\lfloor d(j_1,j_2)\rfloor=\lfloor 3/2\rfloor=1$, (resp. $\lfloor d(j_1,j_2)\rfloor=\lfloor 5/2\rfloor=2$), engaged as the $j$-th entry of $\sigma(n)$, (resp., $\sigma(n')$), where $0\le n<{\mathcal C}_k$, (resp., $0\le n'<{\mathcal C}_k$). We will write $\sigma(n)=\sigma_{k-1}(n)\cdots \sigma_2(n)\sigma_1(n)$.

The growth of the tree $T$ of Dyck nests $F(n)$ will be simplified further from that given in the procedure of Section~\ref{sec4} by recursively updating just one entry of the parent clone $\sigma(\rho(n))$ to obtain the clone $\sigma(n)$.
This uses an equivalence of the set of anchored $k$-blown Dyck nests $F(n)$ and that of their clones $\sigma(n)$, provided in Theorem~\ref{signest}, below.

\begin{table}[htp]
$$
\begin{array}{||r|r||r||}\hline\hline
\;\;\;*&\;\;\;2&\\\hline
*&0&\\
0&-2&\\
&0&\\\hline\hline
3&2&2\\\hline
*&-3&\;\;\;0\\
0&1&-2\\
&0&-3\\
&&0\\\hline\hline
\end{array}
\qquad
\begin{array}{||r|r|r|r|r||r|r|r||r||}\hline\hline
\;\;\;* &2          &\;\;\;3 &\;\;\;2  &\;\;\;2 &       &        &          &\\\hline
\;\;\;*  &\;\;\;0   &\;\;\;0  &-3      &\;\;\;0  &     &          &         &\\
\;\;\;0  &-2       &\;\;\;0  &\;\;\;1 &-2       &      &        &          &\\
          &\;\;\;0  &          &\;\;\;0 &-3      &       &         &        &\\
          &          &          &         &\;\;\;0 &      &         &        &\\\hline\hline
\;\;\;4  &2       &\;\;\;3  &\;\;\;2 &\;\;\;2 &\;\;\;3&\;\;\;2&\;\;\;2&\;\;\;2\\\hline
\;\;\;0  &\;\;\;0 &-4       &\;\;\;1 &\;\;\;0 &\;\;\;0&-3     &-4    &0\\
\;\;\;0  &-2      &\;\;\;0  &-3      &-2      &\;\;\;0&\;\;\;1&\;\;\;2&-2\\
         &\;\;\;0  &           &\;\;\;0 &\;\;\;1 &        &\;\;\;0 &\;\;\;1&-3\\
         &          &           &         &0       &        &         &\;\;\;0&\;\;\;1\\
         &          &           &         &         &       &          &\;\;\;  &\;\;\;0\\\hline\hline
\end{array}$$
$$\begin{array}{|r|r|r|r|r|r|r|r|r|r|r|r|r|r||r|r|r|r||r||}
\hline\hline
*\!&\!2\!&\!3\!&\!2\!&\!2\!&\!4\!&\!2\!&\!3\!&\!2\!&\!2\!&\!3\!&\!2\!&\!2\!&\!2\!&\!\!&\!\!&\!\!&\!\!&\!\\\hline
\;\;\;*\!&\!\;\;\;0\!&\!\;\;\;0\!&\!-3\!&\!0 \!&\!0\!&\!0\!&\!-4\!&\!1\!&\!0\!&\!\;\;\;0\!&\!-3\!&\!-4\!&\!0\!&\!\!&\!\!&\!\!&\!\!&\!\\
\;\;\;0\!&\!-2\!&\!0\!&\!1\!&\!-2\!&\!\;\;\;0\!&\!-2\!&\!0\!&\!-3\!&\!-2\!&\!0\!&\!1\!&\!2\!&\!-2\!&\!\!&\!\!&\!\!&\!\!&\!\\
 \!&\!0  \!&\!  \!&\!0\!&\!-3 \!&\!\!&\!0\!&\! \!&\!0\!&\!1\!&\!\!&\!0\!&\!1\!&\!-3\!&\!\!&\!\!&\!\!&\!\!&\!\\
 \!&\!    \!&\!  \!&\!  \!&\!0 \!&\!\!&\!\!&\! \!&\!\!&\!0\!&\!\!&\!\!&\!0\!&\!-4\!&\!\!&\!\!&\!\!&\!\!&\!\\
\!&\!    \!&\!  \!&\!  \!&\!  \!&\!\!&\!\!&\! \!&\!\!&\!\!&\!\!&\!\!&\!\!&\!0\!&\!\!&\!\!&\!\!&\!\!&\!\\\hline\hline
5\!&\!2\!&\!3\!&\!2\!&\!2\!&\!4\!&\!2\!&\!3\!&\!2\!&\!2\!&\!3\!&\!2\!&\!2\!&\!2\!&\!\!&\!\!&\!\!&\!\!&\!\\\hline
\;\;\;0\!&\!\;\;\;0\!&\!\;\;\;0\!&\!-3\!&\!0 \!&\!     5\!&\! 0\!&\! 1\!&\!-4\!&\!0\!&\!\;\;\;0\!&\!-3\!&\!1\!&\!0\!&\!\!&\!\!&\!\!&\!\!&\!\\
\;\;\;0\!&\!     -2\!&\!      0\!&\!1\!&\!-2\!&\!\;\;\;0\!&\!-2\!&\!0\!&\!2\!&\!-2\!&\!0\!&\!1\!&\!-3\!&\!-2\!&\!\!&\!\!&\!\!&\!\!&\!\\
       \!&\!       0\!&\!        \!&\!0\!&\!-3 \!&\!      \!&\!  0\!&\! \!&\!0\!&\!-4\!&\!\!&\!0\!&\!-4\!&\!-3\!&\!\!&\!\!&\!\!&\!\!&\!\\
       \!&\!         \!&\!        \!&\!  \!&\! 0 \!&\!      \!&\!\!&\! \!&\!\!&\!0\!&\!\!&\!\!&\!0\!&\!1\!&\!\!&\!\!&\!\!&\!\!&\!\\
\!&\!    \!&\!  \!&\!  \!&\!  \!&\!\!&\!\!&\! \!&\!\!&\!\!&\!\!&\!\!&\!\!&\!0\!&\!\!&\!\!&\!\!&\!\!&\!\\\hline\hline
 \!&\! \!&\! \!&\! \!&\! \!&\!4\!&\!2\!&\!3\!&\!2\!&\!2\!&\!3\!&\!2\!&\!2\!&\!2\!&\!3\!&\!2\!&\!2\!&\!2\!&\!2\\\hline
      \!&\!      \!&\!\; \;\; \!&\!  \!&\! \!&\!      0\!&\! 0\!&\!-4\!&\! 1\!&\!0\!&\! -5\!&\! 2\!&\!1\!&\!0\!&\!\;\;\;0\!&\!-3\!&\!-4\!&\!-5\!&\!0\\
      \!&\!       \!&\!       \!&\!  \!&\! \!&\!\;\;\;0\!&\!-2\!&\!0\!&\!-3\!&\!-2\!&\!0 \!&\!-4\!&\!-3\!&\!-2\!&\!0\!&\!1\!&\!2\!&\!3\!&\!-2\\
      \!&\!       \!&\!       \!&\!  \!&\! \!&\!      \!&\!   0\!&\! \!&\!  0\!&\!  1\!&\!  \!&\!0 \!&\! 1\!&\! 2\!&\!\!&\!0\!&\!1\!&\!2\!&\!-3\\
      \!&\!       \!&\!       \!&\!  \!&\!  \!&\!      \!&\!    \!&\! \!&\!    \!&\!  0\!&\!  \!&\!   \!&\!0 \!&\!1\!&\!\!&\!\!&\!0\!&\!1\!&\!-4\\
\!&\!    \!&\!  \!&\!  \!&\!  \!&\!\!&\!\!&\! \!&\!\!&\!\!&\!\!&\!\!&\!\!&\!0\!&\!\!&\!\!&\!\!&\!0\!&\!-5\\
\!&\!    \!&\!  \!&\!  \!&\!  \!&\!\!&\!\!&\! \!&\!\!&\!\!&\!\!&\!\!&\!\!&\!0\!&\!\!&\!\!&\!\!&\!\!&\!0\\\hline\hline
\end{array}$$
\caption{Disposition of threads of $T$}
\label{of}
\end{table}

\begin{obs}\label{obs} In the transformation from $F(\rho(n))$ to $F(n)$ in Section~\ref{sec4}, $k_1k_2$ is present either in $X$ or in $Y$; if $k_1k_2$ is in $X$, then $\sigma_{\gamma(n)}(n)$ depends on $k$ because of blowing, and so $\sigma_{\gamma(n)}(n)=k+h(n)$, for some value $h(n)<0$; if on the contrary $k_1k_2$ is in $Y$, then $\sigma_{\gamma(n)}(n)$ does not depend on $k$, and so $\sigma_{\gamma(n)}(n)=h(n)$, for some $h(n)\ge 0$. In both cases, the remaining entries of $\sigma(n)$ other than $\sigma_{\gamma(n)}(n)$ are kept the same in $F(n)$ as in $F(\rho(n))$.\end{obs}

\begin{theorem}\label{id} For each $n\ne 0$, $\sigma(\rho(n))$ and $\sigma(n)$ differ solely at the $\gamma(n)$-th entry.
 \end{theorem}

 \begin{proof}
 There is a sole difference between the parent $\sigma(\rho(n))$ of $\sigma(n)$ and $\sigma(n)$ itself, occurring at the $\gamma(n)$-th (right-to-left) position; its entry is just increased in one unit from $\sigma(\rho(n))$ to $\sigma(n)$, namely $\sigma_{\gamma(n)}(n)=\sigma_{\gamma(n)}(\rho(n))+1$. The effect of this via the permutation (transposition, or castling) of the inner strings $X$ and $Y$ from $F(\rho(n))=Z|X|Y|W$ to $F(n)=Z|Y|X|W$  (Section~\ref{sec4}) modifies just one of the halfway-distance floors $\sigma_j(n)=\lfloor d(j_1,j_2)/2\rfloor$ between the first appearance, $j_1$,  of the corresponding $j\in[0,k-1]$ in $F(n)$ and its second appearance, $j_2$, namely $\sigma_{\gamma(n)}(n)=\lfloor d(\gamma(n)_1,\gamma(n)_2)/2\rfloor$.
 \end{proof}

Table~\ref{signa} exemplifies how the procedure in Section~\ref{sec4} is solely determined by the integer entries $\sigma_{\gamma(n)}(n)$ of the corresponding permutation $\sigma(n)$ in Theorem~\ref{id}, supplying in the
\begin{enumerate}\item first column: the value of $k$ that depends on the value of $n$ in the third column;
\item second and third columns: the values of the parent index $\rho(n)$ and $n$ itself;
\item fourth and fifth columns: the TRGS $b(n)$ and the value of $\gamma(n)$, respectively;
\item sixth column: $0F(\rho(n))=0|Z|X|Y|W$; with $X$ underlined and $Y$ in Italics;
\item seventh column: $0F(n)=0|Z|Y|X|W$, with $X,Y$ as in the sixth column; 
\item last three columns: the values of $\sigma(\rho(n))$, $\sigma(n)$ and $\sigma_{\gamma(n)}(n)$, this last value given in terms of $k$ via Observation~\ref{obs}.
\end{enumerate}

\begin{theorem}\label{signest}
The correspondence that assigns each $n$-nest to its clone is a bijection.
\end{theorem}

\begin{proof} Let $b(n)=a_{k-1}\ldots a_2a_1$ be a TRGS. The anchored $k$-blown (or tight) Dyck nest $^0\!F(n)=0F(n)=0F_1\cdots F_{2k}$ 
has rightmost entry $F_{2k}=1_2$, so $\sigma_1(n)$ determines the position of $1_1$. For example, if $\sigma_1(n)=0$, then
$F_{2k-1}=1_1$, so $a_1$ is a local maximum. To obtain $0F(n)$ from $\sigma(n)$, we initialize a candidate for $F(n)$ as the $2k$-string $F(n,0)=00\cdots 0$.
Setting successively $1_2,1_1,2_2,2_1,\ldots,(k-1)_2,(k-1)_1$ instead of the zeros of $F(n,0)$ from right to left according to the indications $\sigma_i(n)$, for $i=1,2,\ldots,k-1$, is done in stages $F(n,i-1)\rightarrow F(n,i)$ by setting each pair $(i_1,i_2)$ as an outermost pair, only constrained by the presence of already replaced positions; after setting a value $i_1$ in the initial position, we restart if necessary on the right again with the replacement of the remaining zeros by the remaining pairs $(i_1,i_2)$ in ascending order from right to left.
This allows to recover $F(n)$ from the $\sigma(n)$'s by finally replacing the only resulting substring 00 in $F(n,k-1)$ by $k_1k_2$.
\end{proof}

\noindent We introduce a family $\{\Phi_j| j>0\} $ of subsequences of $\mathbb{N}$ defined by the following properties:

\begin{enumerate}
\item $\gamma(\Phi_1)$ is the subsequence of $\gamma(\mathbb{N})$ formed by all indices $\gamma(n)$ larger than 1.
\item The first term of $\Phi_j$ is 1 if j=1; and $n$ such that $b(n)$ is the smallest TRGS having suffix $(j-1)(j-1)$ for $j>1$.
\item Assume $n\in\Phi_j$, $b(n)=a_{k-1}\cdots a_1$ and either $a_1=0$ or $n'\notin\Phi_j\;\forall n'\in\mathbb{N}$ with
$b(n')=a_{k-1}\cdots a_2a'_1$, ($a'_1<a_1$). Then, $b(n)|j=b(n'')$ has $n''\in\Phi_j\; \forall j\in[0,a_1]$.
If so and $b(m)=a_{k-1}\cdots a_2(a_1+j')$ has $m\in\Phi_j\; \forall j'\ge1$, then $b(m)|j=b(m')$ has $m'\in\Phi_j$.
\item Each braid $Q^\ell=\{P^i\}_{i=u_\ell}^{t_\ell}$ has associated subsequence $\Gamma^\ell=(\iota,2,\ldots,2)$ in $\Gamma=\gamma(\mathbb{N})$, where $\iota>2$. If there exist $z$ penultimate terms $\gamma(n)=2$ in $\Gamma^\ell$ yielding maximal prefixes of common fixed length $y>0$ in the threads of $Q^\ell\setminus P^{t_\ell}$ with associated images in $h(\Phi_j)$ ending at $h(m)$, where $b(m)=a_{k-1}\cdots a_3(y+j)y)$\; $\forall j\in[0,z-1]$, then $m'\in\Phi_j$\; $\forall b(m')=a_{k-1}\cdots a_3(y+z)j'$ with
$j'\in(y,y+z]$, which yields a suffix of $P^{t_\ell}$ with $h$-image $\{h(\alpha_{j'})|j'\in(y,y+z]\}$.
\end{enumerate}

\begin{table}[htp]
$$\begin{array}{||cccccccccccccc||cccccccccccccc||ccccccccc||cccc||c||}\hline\hline\!\!\rm{*}\!&\!\!\rm{2}\!&\!\!\rm{3}\!&\!\!\rm{2}\!&\!\!\rm{2}\!&\!\!\rm{4}\!&\!\!\rm{2}\!&\!\!\rm{3}\!&\!\!\rm{2}\!&\!\!\rm{2}\!&\!\!\rm{3}\!&\!\!\rm{2}\!&\!\!\rm{2}\!&\!\!\rm{2}\!&\!\!\rm{5}\!&\!\!\rm{2}\!&\!\!\rm{3}\!&\!\!\rm{2}\!&\!\!\rm{2}\!&\!\!\rm{4}\!&\!\!\rm{2}\!&\!\!\rm{3}\!&\!\!\rm{2}\!&\!\!\rm{2}\!&\!\!\rm{3}\!&\!\!\rm{2}\!&\!\!\rm{2}\!&\!\!\rm{2}\!&\!\!\rm{4}\!&\!\!\rm{2}\!&\!\!\rm{3}
\!&\!\!\rm{2}\!&\!\!\rm{2}\!&\!\!\rm{3}\!&\!\!\rm{2}\!&\!\!\rm{2}\!&\!\!\rm{2}\!&\!\!\rm{3}\!&\!\!\rm{2}\!&\!\!\rm{2}\!&\!\!\rm{2}\!&\!\!\rm{2}\!\\\hline

\!\!^*\!&\!\!\!\!^\bullet\!\!&\!\!\!\!^\bullet\!\!&\!\!^\bullet\!&\!\!^\circ\!&\!\!^\bullet\!&\!\!^\circ\!&\!\!^\bullet\!&\!\!^\bullet\!&\!\!^\circ\!&\!\!^\circ\!&\!\!^\circ\!&\!\!^\bullet\!&\!\!^\circ\!&\!\!^\bullet\!&\!\!^\circ\!&\!\!^\circ\!&\!\!^\circ\!&\!\!^\circ\!&\!\!^\bullet\!&\!\!^\circ\!&\!\!^\bullet\!&\!\!^\bullet\!&\!\!^\circ\!&\!\!^\circ\!&\!\!^\circ\!&\!\!^\bullet\!&\!\!^\circ\!&\!\!^\circ\!&\!\!^\circ\!&\!\!^\circ\!&\!\!^\circ\!&\!\!^\circ\!&\!\!^\bullet\!&\!\!^\bullet\!&\!\!^\bullet\!&\!\!^\circ\!&\!\!^\circ\!&\!\!^\circ\!&\!\!^\circ\!&\!\!^\bullet\!&\!\!^\circ\!\vspace*{-1.5mm}\\
\!\!^\bullet\!&\!\!^\bullet\!&\!\!^\circ\!&\!\!^\bullet\!&\!\!^\circ\!&\!\!^\circ\!&\!\!^\circ\!&\!\!^\circ\!&\!\!^\bullet\!&\!\!^\circ\!&\!\!^\circ\!&\!\!^\circ\!&\!\!^\bullet\!&\!\!^\circ\!&\!\!^\circ\!&\!\!^\circ\!&\!\!^\circ\!&\!\!^\circ\!&\!\!^\circ\!&\!\!^\circ\!&\!\!^\circ\!&\!\!^\circ\!&\!\!^\bullet\!&\!\!^\circ\!&\!\!^\circ\!&\!\!^\circ\!&\!\!^\bullet\!&\!\!^\circ\!&\!\!^\circ\!&\!\!^\circ\!&\!\!^\circ\!&\!\!^\circ\!&\!\!^\circ\!&\!\!^\circ\!&\!\!^\bullet\!&\!\!^\bullet\!&\!\!^\circ\!&\!\!^\circ\!&\!\!^\circ\!&\!\!^\circ\!&\!\!^\bullet\!&\!\!^\circ\!\vspace*{-1.5mm}\\
\!\!\!&\!\!^\circ\!&\!\!\!&\!\!^\circ\!&\!\!^\bullet\!&\!\!\!&\!\!^\circ\!&\!\!\!&\!\!^\circ\!&\!\!^\bullet\!&\!\!\!&\!\!^\circ\!&\!\!^\bullet\!&\!\!^\circ\!&\!\!\!&\!\!^\circ\!&\!\!\!&\!\!^\circ\!&\!\!^\circ\!&\!\!\!&\!\!^\circ\!&\!\!\!&\!\!^\circ\!&\!\!^\bullet\!&\!\!\!&\!\!^\circ\!&\!\!^\bullet\!&\!\!^\circ\!&\!\!\!&\!\!^\circ\!&\!\!\!&\!\!^\circ\!&\!\!^\circ\!&\!\!\!&\!\!^\circ\!&\!\!^\circ\!&\!\!^\bullet\!&\!\!\!&\!\!^\circ\!&\!\!^\circ\!&\!\!^\bullet\!&\!\!^\circ\!\vspace*{-1.5mm}\\
\!\!\!&\!\!\!&\!\!\!&\!\!\!&\!\!^\circ\!&\!\!\!&\!\!\!&\!\!\!&\!\!\!&\!\!^\circ\!&\!\!\!&\!\!\!&\!\!^\circ\!&\!\!^\bullet\!&\!\!\!&\!\!\!&\!\!\!&\!\!\!&\!\!^\circ\!&\!\!\!&\!\!\!&\!\!\!&\!\!\!&\!\!^\circ\!&\!\!\!&\!\!\!&\!\!^\circ\!&\!\!^\bullet\!&\!\!\!&\!\!\!&\!\!\!&\!\!\!&\!\!^\circ\!&\!\!\!&\!\!\!&\!\!^\circ\!&\!\!^\bullet\!&\!\!\!&\!\!\!&\!\!^\circ\!&\!\!^\bullet\!&\!\!^\circ\!\vspace*{-1.5mm}\\
\!\!\!&\!\!\!&\!\!\!&\!\!\!&\!\!\!&\!\!\!&\!\!\!&\!\!\!&\!\!\!&\!\!\!&\!\!\!&\!\!\!&\!\!\!&\!\!^\circ\!&\!\!\!&\!\!\!&\!\!\!&\!\!\!&\!\!\!&\!\!\!&\!\!\!&\!\!\!&\!\!\!&\!\!\!&\!\!\!&\!\!\!&\!\!\!&\!\!^\circ\!&\!\!\!&\!\!\!&\!\!\!&\!\!\!&\!\!\!&\!\!\!&\!\!\!&\!\!\!&\!\!^\circ\!&\!\!\!&\!\!\!&\!\!\!&\!\!^\circ\!&\!\!^\bullet\!\vspace*{-1.5mm}\\
\!\!\!&\!\!\!&\!\!\!&\!\!\!&\!\!\!&\!\!\!&\!\!\!&\!\!\!&\!\!\!&\!\!\!&\!\!\!&\!\!\!&\!\!\!&\!\!\!&\!\!\!&\!\!\!&\!\!\!&\!\!\!&\!\!\!&\!\!\!&\!\!\!&\!\!\!&\!\!\!&\!\!\!&\!\!\!&\!\!\!&\!\!\!&\!\!\!&\!\!\!&\!\!\!&\!\!\!&\!\!\!&\!\!\!&\!\!\!&\!\!\!&\!\!\!&\!\!\!&\!\!\!&\!\!\!&\!\!\!&\!\!\!&\!\!^\circ\!\\
\hline\hline
\!\!\rm{6}\!&\!\!\rm{2}\!&\!\!\rm{3}\!&\!\!\rm{2}\!&\!\!\rm{2}\!&\!\!\rm{4}\!&\!\!\rm{2}\!&\!\!\rm{3}\!&\!\!\rm{2}\!&\!\!\rm{2}\!&\!\!\rm{3}\!&\!\!\rm{2}\!&\!\!\rm{2}\!&\!\!\rm{2}\!&\!\!\rm{5}\!&\!\!\rm{2}\!&\!\!\rm{3}\!&\!\!\rm{2}\!&\!\!\rm{2}\!&\!\!\rm{4}\!&\!\!\rm{2}\!&\!\!\rm{3}\!&\!\!\rm{2}\!&\!\!\rm{2}\!&\!\!\rm{3}\!&\!\!\rm{2}\!&\!\!\rm{2}\!&\!\!\rm{2}\!&\!\!\rm{4}\!&\!\!\rm{2}\!&\!\!\rm{3}
\!&\!\!\rm{2}\!&\!\!\rm{2}\!&\!\!\rm{3}\!&\!\!\rm{2}\!&\!\!\rm{2}\!&\!\!\rm{2}\!&\!\!\rm{3}\!&\!\!\rm{2}\!&\!\!\rm{2}\!&\!\!\rm{2}\!&\!\!\rm{2}\!\\\hline

\!\!^\circ\!&\!\!^\circ\!&\!\!^\circ\!&\!\!^\circ\!&\!\!^\circ\!&\!\!^\circ\!&\!\!^\circ\!&\!\!^\circ\!&\!\!^\circ\!&\!\!^\circ\!&\!\!^\circ\!&\!\!^\circ\!&\!\!^\circ\!&\!\!^\circ\!&\!\!^\bullet\!&\!\!^\circ\!&\!\!^\circ\!&\!\!^\circ\!&\!\!^\circ\!&\!\!^\bullet\!&\!\!^\circ\!&\!\!^\bullet\!&\!\!^\bullet\!&\!\!^\circ\!&\!\!^\circ\!&\!\!^\circ\!&\!\!^\bullet\!&\!\!^\circ\!&\!\!^\circ\!&\!\!^\circ\!&\!\!^\circ\!&\!\!^\circ\!&\!\!^\circ\!&\!\!^\bullet\!&\!\!^\bullet\!&\!\!^\bullet\!&\!\!^\circ\!&\!\!^\circ\!&\!\!^\circ\!&\!\!^\circ\!&\!\!^\bullet\!&\!\!^\circ\!\vspace*{-1.5mm}\\
\!\!^\circ\!&\!\!^\circ\!&\!\!^\circ\!&\!\!^\circ\!&\!\!^\circ\!&\!\!^\circ\!&\!\!^\circ\!&\!\!^\circ\!&\!\!^\circ\!&\!\!^\circ\!&\!\!^\circ\!&\!\!^\circ\!&\!\!^\circ\!&\!\!^\circ\!&\!\!^\circ\!&\!\!^\circ\!&\!\!^\circ\!&\!\!^\circ\!&\!\!^\circ\!&\!\!^\circ\!&\!\!^\circ\!&\!\!^\circ\!&\!\!^\bullet\!&\!\!^\circ\!&\!\!^\circ\!&\!\!^\circ\!&\!\!^\bullet\!&\!\!^\circ\!&\!\!^\circ\!&\!\!^\circ\!&\!\!^\circ\!&\!\!^\circ\!&\!\!^\circ\!&\!\!^\circ\!&\!\!^\bullet\!&\!\!^\bullet\!&\!\!^\circ\!&\!\!^\circ\!&\!\!^\circ\!&\!\!^\circ\!&\!\!^\bullet\!&\!\!^\circ\!\vspace*{-1.5mm}\\
\!\!\!&\!\!^\circ\!&\!\!\!&\!\!^\circ\!&\!\!^\circ\!&\!\!\!&\!\!^\circ\!&\!\!\!&\!\!^\circ\!&\!\!^\circ\!&\!\!\!&\!\!^\circ\!&\!\!^\circ\!&\!\!^\circ\!&\!\!\!&\!\!^\circ\!&\!\!\!&\!\!^\circ\!&\!\!^\circ\!&\!\!\!&\!\!^\circ\!&\!\!\!&\!\!^\circ\!&\!\!^\bullet\!&\!\!\!&\!\!^\circ\!&\!\!^\bullet\!&\!\!^\circ\!&\!\!\!&\!\!^\circ\!&\!\!\!&\!\!^\circ\!&\!\!^\circ\!&\!\!\!&\!\!^\circ\!&\!\!^\circ\!&\!\!^\bullet\!&\!\!\!&\!\!^\circ\!&\!\!^\circ\!&\!\!^\bullet\!&\!\!^\circ\!\vspace*{-1.5mm}\\
\!\!\!&\!\!\!&\!\!\!&\!\!\!&\!\!^\circ\!&\!\!\!&\!\!\!&\!\!\!&\!\!\!&\!\!^\circ\!&\!\!\!&\!\!\!&\!\!^\circ\!&\!\!^\circ\!&\!\!\!&\!\!\!&\!\!\!&\!\!\!&\!\!^\circ\!&\!\!\!&\!\!\!&\!\!\!&\!\!\!&\!\!^\circ\!&\!\!\!&\!\!\!&\!\!^\circ\!&\!\!^\bullet\!&\!\!\!&\!\!\!&\!\!\!&\!\!\!&\!\!^\circ\!&\!\!\!&\!\!\!&\!\!^\circ\!&\!\!^\bullet\!&\!\!\!&\!\!\!&\!\!^\circ\!&\!\!^\bullet\!&\!\!^\circ\!\vspace*{-1.5mm}\\
\!\!\!&\!\!\!&\!\!\!&\!\!\!&\!\!\!&\!\!\!&\!\!\!&\!\!\!&\!\!\!&\!\!\!&\!\!\!&\!\!\!&\!\!\!&\!\!^\circ\!&\!\!\!&\!\!\!&\!\!\!&\!\!\!&\!\!\!&\!\!\!&\!\!\!&\!\!\!&\!\!\!&\!\!\!&\!\!\!&\!\!\!&\!\!\!&\!\!^\circ\!&\!\!\!&\!\!\!&\!\!\!&\!\!\!&\!\!\!&\!\!\!&\!\!\!&\!\!\!&\!\!^\circ\!&\!\!\!&\!\!\!&\!\!\!&\!\!^\circ\!&\!\!^\bullet\!\vspace*{-1.5mm}\\
\!\!\!&\!\!\!&\!\!\!&\!\!\!&\!\!\!&\!\!\!&\!\!\!&\!\!\!&\!\!\!&\!\!\!&\!\!\!&\!\!\!&\!\!\!&\!\!\!&\!\!\!&\!\!\!&\!\!\!&\!\!\!&\!\!\!&\!\!\!&\!\!\!&\!\!\!&\!\!\!&\!\!\!&\!\!\!&\!\!\!&\!\!\!&\!\!\!&\!\!\!&\!\!\!&\!\!\!&\!\!\!&\!\!\!&\!\!\!&\!\!\!&\!\!\!&\!\!\!&\!\!\!&\!\!\!&\!\!\!&\!\!\!&\!\!^\circ\!\\\hline\hline

\!\!\!&\!\!\!&\!\!\!&\!\!\!&\!\!\!&\!\!\!&\!\!\!&\!\!\!&\!\!\!&\!\!\!&\!\!\!&\!\!\!&\!\!\!&\!\!\!&\!\!\rm{5}\!&\!\!\rm{2}\!&\!\!\rm{3}\!&\!\!\rm{2}\!&\!\!\rm{2}\!&\!\!\rm{4}\!&\!\!\rm{2}\!&\!\!\rm{3}\!&\!\!\rm{2}\!&\!\!\rm{2}\!&\!\!\rm{3}\!&\!\!\rm{2}\!&\!\!\rm{2}\!&\!\!\rm{2}\!&\!\!\rm{4}\!&\!\!\rm{2}\!&\!\!\rm{3}
\!&\!\!\rm{2}\!&\!\!\rm{2}\!&\!\!\rm{3}\!&\!\!\rm{2}\!&\!\!\rm{2}\!&\!\!\rm{2}\!&\!\!\rm{3}\!&\!\!\rm{2}\!&\!\!\rm{2}\!&\!\!\rm{2}\!&\!\!\rm{2}\!\\\hline

\!\!\!&\!\!\!\!\!\!&\!\!\!\!\!\!&\!\!\!&\!\!\!&\!\!\!&\!\!\!&\!\!\!&\!\!\!&\!\!\!&\!\!\!&\!\!\!&\!\!\!&\!\!\!&\!\!^\circ\!&\!\!^\circ\!&\!\!^\circ\!&\!\!^\circ\!&\!\!^\circ\!&\!\!^\circ\!&\!\!^\circ\!&\!\!^\circ\!&\!\!^\circ\!&\!\!^\circ\!&\!\!^\circ\!&\!\!^\circ\!&\!\!^\circ\!&\!\!^\circ\!&\!\!^\bullet\!&\!\!^\circ\!&\!\!^\bullet\!&\!\!^\bullet\!&\!\!^\circ\!&\!\!^\bullet\!&\!\!^\bullet\!&\!\!^\circ\!&\!\!^\circ\!&\!\!^\circ\!&\!\!^\circ\!&\!\!^\bullet\!&\!\!^\bullet\!&\!\!^\circ\!\vspace*{-1.5mm}\\
\!\!\!&\!\!\!&\!\!\!&\!\!&\!\!\!&\!\!\!&\!\!\!&\!\!\!&\!\!\!&\!\!\!&\!\!\!&\!\!\!&\!\!\!&\!\!\!&\!\!^\circ\!&\!\!^\circ\!&\!\!^\circ\!&\!\!^\circ\!&\!\!^\circ\!&\!\!^\circ\!&\!\!^\circ\!&\!\!^\circ\!&\!\!^\circ\!&\!\!^\circ\!&\!\!^\circ\!&\!\!^\circ\!&\!\!^\circ\!&\!\!^\circ\!&\!\!^\circ\!&\!\!^\circ\!&\!\!^\circ\!&\!\!^\circ\!&\!\!^\circ\!&\!\!^\circ\!&\!\!^\bullet\!&\!\!^\circ\!&\!\!^\circ\!&\!\!^\circ\!&\!\!^\circ\!&\!\!^\bullet\!&\!\!^\bullet\!&\!\!^\circ\!\vspace*{-1.5mm}\\
\!\!\!&\!\!\!&\!\!\!&\!\!&\!\!\!&\!\!\!&\!\!\!&\!\!\!&\!\!\!&\!\!\!&\!\!\!&\!\!\!&\!\!\!&\!\!\!&\!\!\!&\!\!^\circ\!&\!\!\!&\!\!^\circ\!&\!\!^\circ\!&\!\!\!&\!\!^\circ\!&\!\!\!&\!\!^\circ\!&\!\!^\circ\!&\!\!\!&\!\!^\circ\!&\!\!^\circ\!&\!\!^\circ\!&\!\!\!&\!\!^\circ\!&\!\!\!&\!\!^\circ\!&\!\!^\circ\!&\!\!\!&\!\!^\circ\!&\!\!^\bullet\!&\!\!^\circ\!&\!\!\!&\!\!^\circ\!&\!\!^\bullet\!&\!\!^\bullet\!&\!\!^\circ\!\vspace*{-1.5mm}\\
\!\!\!&\!\!\!&\!\!\!&\!\!\!&\!\!\!&\!\!\!&\!\!\!&\!\!\!&\!\!\!&\!\!\!&\!\!\!&\!\!\!&\!\!\!&\!\!\!&\!\!\!&\!\!\!&\!\!\!&\!\!\!&\!\!^\circ\!&\!\!\!&\!\!\!&\!\!\!&\!\!\!&\!\!^\circ\!&\!\!\!&\!\!\!&\!\!^\circ\!&\!\!^\circ\!&\!\!\!&\!\!\!&\!\!\!&\!\!\!&\!\!^\circ\!&\!\!\!&\!\!\!&\!\!^\circ\!&\!\!^\circ\!&\!\!\!&\!\!\!&\!\!^\circ\!&\!\!^\circ\!&\!\!^\bullet\!\vspace*{-1.5mm}\\
\!\!\!&\!\!\!&\!\!\!&\!\!\!&\!\!\!&\!\!\!&\!\!\!&\!\!\!&\!\!\!&\!\!\!&\!\!\!&\!\!\!&\!\!\!&\!\!\!&\!\!\!&\!\!\!&\!\!\!&\!\!\!&\!\!\!&\!\!\!&\!\!\!&\!\!\!&\!\!\!&\!\!\!&\!\!\!&\!\!\!&\!\!\!&\!\!^\circ\!&\!\!\!&\!\!\!&\!\!\!&\!\!\!&\!\!\!&\!\!\!&\!\!\!&\!\!\!&\!\!^\circ\!&\!\!\!&\!\!\!&\!\!\!&\!\!^\circ\!&\!\!^\bullet\!\vspace*{-1.5mm}\\
\!\!\!&\!\!\!&\!\!\!&\!\!\!&\!\!\!&\!\!\!&\!\!\!&\!\!\!&\!\!\!&\!\!\!&\!\!\!&\!\!\!&\!\!\!&\!\!\!&\!\!\!&\!\!\!&\!\!\!&\!\!\!&\!\!\!&\!\!\!&\!\!\!&\!\!\!&\!\!\!&\!\!\!&\!\!\!&\!\!\!&\!\!\!&\!\!\!&\!\!\!&\!\!\!&\!\!\!&\!\!\!&\!\!\!&\!\!\!&\!\!\!&\!\!\!&\!\!\!&\!\!\!&\!\!\!&\!\!\!&\!\!\!&\!\!^\circ\!\\\hline\hline

\!\!\!&\!\!\!&\!\!\!&\!\!\!&\!\!\!&\!\!\!&\!\!\!&\!\!\!&\!\!\!&\!\!\!&\!\!\!&\!\!\!&\!\!\!&\!\!\!&\!\!\!&\!\!\!&\!\!\!&\!\!\!&\!\!\!&\!\!\!&\!\!\!&\!\!\!&\!\!\!&\!\!\!&\!\!\!&\!\!\!&\!\!\!&\!\!\!&\!\!\rm{4}\!&\!\!\rm{2}\!&\!\!\rm{3}
\!&\!\!\rm{2}\!&\!\!\rm{2}\!&\!\!\rm{3}\!&\!\!\rm{2}\!&\!\!\rm{2}\!&\!\!\rm{2}\!&\!\!\rm{3}\!&\!\!\rm{2}\!&\!\!\rm{2}\!&\!\!\rm{2}\!&\!\!\rm{2}\!\\\hline

\!\!\!&\!\!\!\!\!\!&\!\!\!\!\!\!&\!\!\!&\!\!\!&\!\!\!&\!\!\!&\!\!\!&\!\!\!&\!\!\!&\!\!\!&\!\!\!&\!\!\!&\!\!\!&\!\!\!&\!\!\!&\!\!\!&\!\!\!&\!\!\!&\!\!\!&\!\!\!&\!\!\!&\!\!\!&\!\!\!&\!\!\!&\!\!\!&\!\!\!&\!\!\!&\!\!^\circ\!&\!\!^\circ\!&\!\!^\circ\!&\!\!^\circ\!&\!\!^\circ\!&\!\!^\circ\!&\!\!^\circ\!&\!\!^\circ\!&\!\!^\circ\!&\!\!^\bullet\!&\!\!^\bullet\!&\!\!^\bullet\!&\!\!^\bullet\!&\!\!^\circ\!\vspace*{-1.5mm}\\
\!\!\!&\!\!\!&\!\!\!&\!\!&\!\!\!&\!\!\!&\!\!\!&\!\!\!&\!\!\!&\!\!\!&\!\!\!&\!\!\!&\!\!\!&\!\!\!&\!\!\!&\!\!\!&\!\!\!&\!\!\!&\!\!\!&\!\!\!&\!\!\!&\!\!\!&\!\!\!&\!\!\!&\!\!\!&\!\!\!&\!\!\!&\!\!\!&\!\!^\circ\!&\!\!^\circ\!&\!\!^\circ\!&\!\!^\circ\!&\!\!^\circ\!&\!\!^\circ\!&\!\!^\circ\!&\!\!^\circ\!&\!\!^\circ\!&\!\!^\circ\!&\!\!^\bullet\!&\!\!^\bullet\!&\!\!^\bullet\!&\!\!^\circ\!\vspace*{-1.5mm}\\
\!\!\!&\!\!\!&\!\!\!&\!\!&\!\!\!&\!\!\!&\!\!\!&\!\!\!&\!\!\!&\!\!\!&\!\!\!&\!\!\!&\!\!\!&\!\!\!&\!\!\!&\!\!\!&\!\!\!&\!\!\!&\!\!\!&\!\!\!&\!\!\!&\!\!\!&\!\!\!&\!\!\!&\!\!\!&\!\!\!&\!\!\!&\!\!\!&\!\!\!&\!\!^\circ\!&\!\!\!&\!^\circ\!\!&\!^\circ\!\!&\!\!\!&\!\!^\circ\!&\!\!^\circ\!&\!^\circ\!\!&\!\!\!&\!\!^\circ\!&\!\!^\circ\!&\!\!^\circ\!&\!\!^\bullet\!\vspace*{-1.5mm}\\
\!\!\!&\!\!\!&\!\!\!&\!\!\!&\!\!\!&\!\!\!&\!\!\!&\!\!\!&\!\!\!&\!\!\!&\!\!\!&\!\!\!&\!\!\!&\!\!\!&\!\!\!&\!\!\!&\!\!\!&\!\!\!&\!\!\!&\!\!\!&\!\!\!&\!\!\!&\!\!\!&\!\!\!&\!\!\!&\!\!\!&\!\!\!&\!\!\!&\!\!\!&\!\!\!&\!\!\!&\!\!\!&\!\!^\circ\!&\!\!\!&\!\!\!&\!\!^\circ\!&\!\!^\circ\!&\!\!\!&\!\!\!&\!\!^\circ\!&\!\!^\circ\!&\!\!^\bullet\!\vspace*{-1.5mm}\\
\!\!\!&\!\!\!&\!\!\!&\!\!\!&\!\!\!&\!\!\!&\!\!\!&\!\!\!&\!\!\!&\!\!\!&\!\!\!&\!\!\!&\!\!\!&\!\!\!&\!\!\!&\!\!\!&\!\!\!&\!\!\!&\!\!\!&\!\!\!&\!\!\!&\!\!\!&\!\!\!&\!\!\!&\!\!\!&\!\!\!&\!\!\!&\!\!\!&\!\!\!&\!\!\!&\!\!\!&\!\!\!&\!\!\!&\!\!\!&\!\!\!&\!\!\!&\!\!^\circ\!&\!\!\!&\!\!\!&\!\!\!&\!\!^\circ\!&\!\!^\bullet\!\vspace*{-1.5mm}\\
\!\!\!&\!\!\!&\!\!\!&\!\!\!&\!\!\!&\!\!\!&\!\!\!&\!\!\!&\!\!\!&\!\!\!&\!\!\!&\!\!\!&\!\!\!&\!\!\!&\!\!\!&\!\!\!&\!\!\!&\!\!\!&\!\!\!&\!\!\!&\!\!\!&\!\!\!&\!\!\!&\!\!\!&\!\!\!&\!\!\!&\!\!\!&\!\!\!&\!\!\!&\!\!\!&\!\!\!&\!\!\!&\!\!\!&\!\!\!&\!\!\!&\!\!\!&\!\!\!&\!\!\!&\!\!\!&\!\!\!&\!\!\!&\!\!^\circ\!\\\hline\hline
\end{array}$$
\caption{Symbolic continuation of Table~\ref{of}}
\label{off}
\end{table}

\begin{theorem}\label{thm1}
If either $b(n)=10\cdots 0$ or $b(n)$ is the endvertex of a thread, then $h(n)=0\in{\mathcal S}$.
\end{theorem}

\begin{proof} The first case in the hypothesis insures the presence of all feasible substrings $j_1j_2$ in $F(n)$. The second case insures at least the presence of the substring $1_11_2$ in $F(n)$.
\end{proof}

\begin{theorem}\label{thm2}
Let $n\in\mathbb{N}$. Then, $\exists$ $r>n$ such that $b(r)=1|b(n)$ and
\begin{enumerate}
\item if $h(n)\in\Phi_1$, then $h(r)\in\Phi_1$ and $h(r)=k-h(n)$;
\item if $h(n)\notin\Phi_1$, then $h(r)\notin\Phi_1$ and $h(r)=h(n)$.
\end{enumerate}
\end{theorem}

\begin{proof}
Item 1 occurs exactly when the substring $k_1k_2$ in $F(n)$ changes position from one side of $1_1$ to the opposite side in $F(r)$, in the procedure of Section~\ref{sec4} starting at $b(\rho(n))$ and $b(\rho(r))$ and ending at $b(n)$ and $b(r)$, respectively. Otherwise, item 2 happens.
\end{proof}

\begin{table}[htp]
$$\begin{array}{|l|}\hline
^{\xi_2^2=2_1|1_1|1_2;}
_{\xi_2^3=2_2|1_1|1_2|1_3;}\\
^{\xi_2^4=2_3|1_1|1_2|1_3|1_4;}
_{\xi_2^5=2_4|1_1|1_2|1_3|1_4|1_5;}\\
^{\cdots;}
_{\xi_3^3=3_1|1_1|\xi_2^2|\xi_2^3=3_1|1_1|2_11_11_2|2_21_11_21_3;}\\
^{\xi_3^4=3_2|1_1|\xi_2^2|\xi_2^3|\xi_2^4=3_2|1_1|2_11_11_2|2_21_11_21_3|2_31_11_21_31_4;}
_{\xi_3^5=3_3|1_1|\xi_2^2|\xi_2^3|\xi_2^4|\xi_2^5=3_2|1_1|2_11_11_2|2_21_11_21_3|2_31_11_21_31_4|2_41_11_21_31_41_5;}\\
^{\cdots;}
_{\xi_4^4=4_1|1_1|\xi_2^2|\xi_3^3|\xi_3^4=4_1|1_1|2_11_11_2|3_11_12_11_11_22_21_11_21_3|3_21_12_11_11_2|2_21_11_21_3|2_31_11_21_31_4;}\\
^{\xi_4^5=4_2|1_1|\xi_2^2|\xi_3^3|\xi_3^4|\xi_3^5;}
_{\cdots}\\
^{\xi_5^5=5_1|1_1|\xi_2^2|\xi_3^3|\xi_4^4|\xi_4^5;}
_{\xi_5^6=5_2|1_1|\xi_2^2|\xi_3^3|\xi_4^4|\xi_4^5|\xi_4^6;}\\
^{\cdots}
_{\xi_a^{a+b}=a_{a+b}|1_1|\xi_2^2|\cdots|\xi_{a-1}^{a-1}|\xi_{a-1}^a|\cdots|\xi_{a-1}^{a+b}, \; \forall 0<a\in\mathbb{N}, \forall 0<b\in\mathbb{N}.}\\\hline
\end{array}$$
\caption{Introduction of strings $\xi_\gamma^b$, for all pairs $(\gamma,b)\in\mathbb{N}^2$ such that $1<\gamma\le b$.}
\label{tab1}
\end{table}

\begin{example} In the upper-left box of Table~\ref{of}, the values $h(\cdot)$ for the five initial threads $P^i$ ($i=0,1,2,3,4$) are disposed in order to start illustrating Theorems~\ref{thm1}-\ref{thm2}.
Each such thread is shown with a singly underlined heading containing the value $\gamma(\cdot)>1$ of its initial vertex, and in the entries below the heading, 
the subsequent values $h(\cdot)$ for the vertices of the thread. In this upper-left box, the braid $Q^0=P^0\cup P^1$ is over the braid $Q^1\setminus P^4=P^2\cup P^3$. The final column for $P^4$ has its values $h(\cdot)$ also determined by the Theorems~\ref{thm1} and~\ref{thm2}.

In the upper-right box of Table~\ref{of}, similar changes are observed with $Q^0\cup Q^1$ on top of $Q^2\cup Q^3$, and to the right of them, the columns of $Q^4$
for threads $P^{10},P^{11},P^{12},P^{13}$ have their values $h(\cdot)$ determined.

The lower box of Table~\ref{of} shows an upper level composed by threads $P^0,\ldots,P^{13}$, a middle level composed by threads $P^{14},\ldots,P^{25}$ and a  lower level composed by threads $P^{26},\ldots,P^{41}$, disposed vertically as to facilitate verifying our assertions.

Similarly with Table~\ref{off}, where the threads $P^0,\ldots,P^{126}$ are disposed in four levels with their underlined headings $\gamma(\cdot)$ as above and the vertically disposed values of $h(\cdot)$ replaced  by rings ``$\circ$" if $h(\cdot)\notin\Phi_1$ and bullets ``$\bullet$'' if $h(\cdot)\in\Phi_1$. The final five threads here are shown in the box on the right hand side of Table~\ref{quilombo} with the data disposed as in Table~\ref{of}.
\end{example}

\begin{theorem}\label{thm3}
Let $1<j\le k\in\mathbb{N}$ and let $n\in\mathbb{N}$ be such that $b(n)=1\cdots(j-1)(j-1)a_{k-j-1}\cdots a_1$. Then, $\exists r>n$ such that
$b(r)=1\cdots(j-1)ja_{k-j-1}\cdots a_1$ is a vertex of $T$ and
\begin{enumerate}
\item if $h(n)\in\Phi_j$, then $h(r)\in\Phi_j$ and $h(r)=k-h(n)$;
\item if $h(n)\notin\Phi_j$, then $h(r)\notin\Phi_j$ and $h(r)=h(n)$.
\end{enumerate}
\end{theorem}

\begin{proof} The argument in the proof of Theorem~\ref{thm2} applies here, too.\end{proof}

\begin{corollary}\label{cor1} Let $n\in\mathbb{N}$ be such that $b(n)=11a_{k-3}\cdots a_1$. Then, $\exists m>n$ such that $b(m)=12a_{k-3}\cdots a_1$ and
\begin{enumerate}
\item if $h(n)\in\Phi_2$, then $h(m)\in\Phi_2$ and $h(m)=k-h(n)$;
\item if $h(n)\notin\Phi_2$, then $h(m)\notin\Phi_2$ and $h(m)=h(n)$.
\end{enumerate}
\end{corollary}
\begin{example}
Applying Corollary~\ref{cor1} to $n=3,7,8,9$, so that $b(n)=11,110,111,112$, and $h(n)=-2,-3,1,0\in\Phi_2$ yields $r=4,10,11,12$, with $b(r)=12,120,121,122$ and  $h(r)=0,0,-2,-3$.
\end{example}

\begin{corollary}\label{cor2}
Let $n\in\mathbb{N}$ be such that $b(n)=122a_{k-4}\cdots a_1$. Then, $\exists r>n$ such that $b(r)=123a_{k-4}\cdots a_1$ and
\begin{enumerate}
\item if $h(n)\in\Phi_3$, then $h(r)\in\Phi_3$ and $h(r)=k-h(n)$;
\item if $h(n)\notin\Phi_3$, then $h(r)\notin\Phi_3$ and $h(r)=h(n)$.
\end{enumerate}
\end{corollary}

\begin{example}
Applying Corollary~\ref{cor2} to $n=12,33,34,35,36$, so $b(n)=122,1220,1221,1222$, $1223$ and $h(\alpha_3)=-3,-4,2,1,0\in\Phi_3$ yields $r=13,37,38,39,40$ with $b(r)=123,1230,1231$, $1232,1233$ and $h(\alpha'_3)=0,0,-2,-3,-4$.
\end{example}

\section{Controlling odd and middle-levels graphs via T}\label{conv}

We introduce strings $\xi_\gamma^b$, for all pairs $(\gamma,b)\in\mathbb{N}^2$ with $1<\gamma\le b$. The entries of each $\xi_\gamma^b$ are integer pairs $(\alpha,\beta)$, denoted $\alpha_\beta$, starting with $\alpha_\beta=1_1$, initial case of the more general notation $1_\beta$, for $\beta\ge 1$. The strings $\xi_\gamma^b$  are determined following Table~\ref{tab1}. The components $\alpha$ in the entries $\alpha_\beta$ represent the indices $\gamma=\gamma(b(n))$ (see Section~\ref{s2&half}) in their order of appearance in ${\mathcal S}$, and $\beta$ is an indicator to distinguish different entries $\alpha_\beta$ with $\alpha$ locally constant.

Next, consider the infinite string $J$ of integer pairs $\alpha_\beta$ formed as the concatenation
\begin{align}\label{J(2)} J=\xi_1^1|\xi_2^2|\cdots|\xi_\gamma^\gamma|\cdots=*|1_1|\xi_2^2|\cdots|\xi_\gamma^\gamma|\cdots,\end{align}
with $\xi_1^1=*|1_1$ standing for the first two lines of Tables~\ref{dia23} and~\ref{signa}, where $*$, representing the root of $T$, stands for the first such line, and $\xi_1^1$ for the second line.

A {\it partition of a string} $A$ is a sequence of substrings $\sigma_1,\sigma_2,\ldots,\sigma_n$ whose concatenation $\sigma_1|\sigma_2|\cdots|\sigma_n$ is equal to $A$.

 We recur to {\it Catalan's reversed triangle} $\Delta'$, whose lines are obtained from Catalan's triangle $\Delta$ (see \cite{D2}) by reversing its lines, so that
 they may be written as in Table~\ref{tab4} that shows the first eight lines $\Delta'_k$ of $\Delta'$, for $k\in[0,7]$. 
Each $\xi_\gamma^b$ in the statement of Theorem~\ref{alfin} is presented in the Tables~\ref{dia23}-\ref{signa} in $\gamma(n)$-columnwise disposition.

\begin{theorem}\label{alfin} Let $k>0$.
The vertex subset of $T$ representing the $\mathbb{Z}_{2k+1}$-classes of $V(O_k)$ and the $\mathbb{D}_{2k+1}$-classes of $V(M_k)$ is given in $J=*|\gamma({\mathcal S}\setminus\{b(0)\})=*|\Gamma=*|\gamma(\mathbb{N})$ by the prefix $$\xi_k^k=\xi_1^1|\xi_2^2|\cdots |\xi_{k-1}^{k-1}|\xi_{k-1}^k,\mbox{ with partition }\{\xi_1^1,\xi_2^2,\ldots,\xi_{k-1}^{k-1},\xi_{k-1}^k\},$$ further refined by splitting the last column $\xi_{k-2}^k$ of $\xi_{k-1}^k$ into the sets $\xi_{k-2}^k$ of its first $k-1$ entries and $X_{k-2}^k$ of its last entry, $a_{k-1}a_{k-2}\cdots a_1=12\cdots (k-1)$. The sizes $|\xi_1^1|$, $|\xi_2^2|$, $\ldots$, $|\xi_{k-1}^{k-1}|$, $|\xi_{k-2}^k|$, $|X_{k-2}^k|$ form the line $\Delta'_{k-1}$ of $\Delta'$.
\end{theorem}

\begin{table}[htp]
$$\begin{array}{|c|c|c|c|c|c|c|c|c|}\hline
&&&&&&&&\tau_0^0=1\\
&&&&&&&\tau_1^1=1&\tau_0^1=1\\
&&&&&&\tau_2^2=2&\tau_1^2=2&\tau_0^2=1\\
&&&&&\tau_3^3=5&\tau_2^3=5&\tau_1^3=3&\tau_0^3=1\\
&&&&\tau_4^4=14&\tau_3^4=14&\tau_2^4=9&\tau_1^4=4&\tau_0^4=1\\
&&&\tau_5^5=42&\tau_4^5=42&\tau_3^5=28&\tau_2^5=14&\tau_1^5=5&\tau_0^5=1\\
&&\tau_6^6=132&\tau_5^6=132&\tau_4^6=90&\tau_3^6=48&\tau_2^6=20&\tau_1^6=6&\tau_0^6=1\\
&\tau_7^7=429&\tau_6^7=429&\tau_5^7=297&\tau_4^7=165&\tau_3^7=75&\tau_2^7=27&\tau_1^7=7&\tau_0^7=1\\
\cdots\cdots&\cdots\cdots&\cdots\cdots&\cdots\cdots&\cdots\cdots&\cdots\cdots&\cdots\cdots&\cdots\cdots&\cdots\cdots\\\hline
\end{array}$$
\caption{An initial detailed portion of Catalan's reversed triangle $\Delta'$.}
\label{tab4}
\end{table}

\begin{proof}  
The statement represents the set of vertices of the induced truncated tree\\ $T[{\mathcal S}\cap b([0,{\mathcal C}_k-1])]$, ($1\le k\in\mathbb{N}$)
via the prefix $\xi_k^k$ of $J$ and the line $\Delta'_{k-1}$ of $\Delta'$. \end{proof}

\begin{theorem}\label{fidel}
The sequence $h({\mathcal S}\!\setminus\!b(0))$ can be recreated by stepwise generation of the induced truncated trees $T[{\mathcal S}\cap b([0,{\mathcal C}_k-1])]$, ($1\le k\in\mathbb{N}$). In the $k$-th step, the determinations specified in {\rm Theorems~\ref{thm1}-\ref{thm2}} are performed in the natural order of the TRGS's.
The $k$-step completes those determinations, namely $(n,h(n))\rightarrow(r,h(r))$,
for the lines of $\Delta'$ corresponding to the sets $\xi_j^j$  ($j=1,\ldots,k-1$),
and ends up with the determinations $(n,h(n))\rightarrow(r,h(r))$ for $j=k$ in the line corresponding to $\xi_{k-2}^k$ and $(n,h(n))\rightarrow(r,h(r))$ in the final line for $j=k+1$, corresponding to $X_{k-2}^k$.
\end{theorem}

\begin{proof}
Theorem~\ref{fidel} is used to express the stepwise nature of the generation of the sequence $h({\mathcal S}\!\setminus\!b(0))$. The methodology in the statement is obtained by integrating steps applying Theorem~\ref{thm3} in the way prescribed, that yields the correspondence with the lines of $\Delta'$.\end{proof}

\begin{corollary}\label{coro} Let $v\in V(O_n)$ (resp., $v\in V(M_k)$).
The sequence of quadruples \begin{eqnarray}\label{13}(b,\gamma,\rho,h)({\mathbb N})=(b({\mathbb N}),\gamma({\mathbb N}),\rho({\mathbb N}),h({\mathbb N}))\end{eqnarray} allows to retrieve $v$ by locating either its oriented $(2k+1)$- (resp., $(4k+2)$-) cycle in the uniform 2-factor ({\rm Section~\ref{1f}}) or in a specific ${\mathbb Z}_{2k+1}$- (resp., ${\mathbb D}_{2k+1}$-) class ({\rm Sections~\ref{section6}-\ref{section8}}) and then locating $v$ in such cycle or class by departing from its only anchored Dyck word. The sequence {\rm(\ref{13})} allows to enlist all vertices $v$ by ordering their cycles or classes, including all vertices in each such cycle or class, starting with its anchored Dyck word.
\end{corollary}

\begin{proof}
Let $n\in{\mathbb N}$. Then, $\gamma(n)$ yields the required update location in the TRGS $b(n)\in{\mathcal S}$ with respect to the parent TRGS $b(\rho(n))\in{\mathcal S}$, while $h(n)$ yields the specific update, as determined in Theorems~\ref{alfin}-\ref{fidel}. This produces the corresponding clone. Then, Theorem~\ref{id} allows to recover the original Dyck word from that clone, and thus the corresponding vertex of $O_k$ (resp., $M_k$) by local translation in its containing cycle in the cycle factor of Section~\ref{1f}, or cyclic (resp., dihedral) class (as pointed out in Section~\ref{2f}).
\end{proof}

\section{Hamilton cycles in odd and middle-levels graphs}\label{Ham}

A {\it flippable tuple} \cite{Hcs} in a vertical list $L(n)$ is a pair $FT(n,j)$ of contiguous lines in $L(n)$ having its $k$-supplementary entry pair at the $j$-th position counted from the right ($j\in[0,2k]$). Let $2<\kappa\in\mathbb{Z}$.  A {\it flipping $\kappa$-cycle} \cite{Hcs} is a finite sequence of pairwise different vertical lists $L(n_j)$, ($n=1,\ldots,\kappa)$ determining a $2\kappa$-cycle in $O_k$ containing successive pairwise disjoint edges whose endvertex pairs $\{\chi_0^j,\chi_1^j\}$ are flippable tuples in their corresponding vertical lists $L(n_j)$ ($n=1,\ldots,\kappa$), with the
vertical pairs of number $k$-supplementary entries happening at pairwise different coordinate positions.

\begin{example}\label{k=3-310} As mentioned for $k=3$,  the five right columns in Table~\ref{uniform} contain the lists $L(n)$ ($n\in[0,4]$). These lists contain in those five final columns additional information that allows to assemble the flipping triples $\tau_0=(L(0),L(1),L(2))$ and $\tau_1=(L(0),L(3),L(4))$, with the initial line $0F^3(n)$ of each such $L(n)$ having a sole underlined entry per flippable tuple corresponding to the underlined entries of its two constituent contiguous $(2k+1)$-tuples, say  $\chi_0,\chi_1$. Here, $\chi_1$ is provided with a superindex containing: {\bf(i)} an index $z\in\{0,1\}$ relating to a sole associated triple $\tau_z$ ($z\in\{0,1\}$); {\bf(ii)} a vertical arrow indicating a definite orientation of the edge $\chi_0\chi_1$ which determines an arc $\chi'\chi''$ with $\{\chi_0,\chi_1\}=\{\chi',\chi''\}$; {\bf(iii)} an index $n'$ such that $L(n')$ is in $\tau_z$ and contains a flippable tuple determining an arc $\chi'''\chi''''$ so that the arc $\chi''\chi'''$ is in $\tau_z$. By considering the three flippable tuples obtained in this way and the additional neighbor adjacencies, an oriented 6-cycle is obtained.
The triangles $\tau_0,\tau_1$ form the two hyperedges of a connected acyclic hypergraph on the vertex set $\{L(n)|i\in[0,4]\}$ that yields the simplest case of the construction of a Hamilton cycle in the odd graphs, in this case in $O_3$. Each of $\tau_0$ and $\tau_1$ yields a 21-cycle in $O_3$ by means of symmetric differences. The presence of $L(0)$ in both $\tau_0$ and $\tau_1$ then allows to transform both 21-cycles into the claimed Hamilton cycle of $O_3$ by means of corresponding flippable tuples in $L(0)$, one for $\tau_0$ and the other one for $\tau_1$.
\end{example}

Consider the following Dyck-word collections (triples, quadruples, etc.):
\begin{eqnarray}\label{!}\begin{array}{lllllr}
S_1(w)&=&\{\xi_{1^w}^1=0w001\ul{1}1,&\xi_{1^w}^2=0w\ul{0}1101,&\xi_{1^w}^3=0w0101\ul{1}&\},\\
S_2&=& \{\xi_2^1\hspace{2.0mm}=0\ul{0}110011,&\xi_2^2\hspace{2.0mm}=0010011\ul{1},&\xi_2^3\hspace{2.0mm}=00010\ul{1}11&\},\\
S_3&=& \{\xi_3^1\hspace{2.0mm}=00011\ul{1},&\xi_3^2\hspace{2.0mm}=0100\ul{1}1,&\xi_3^3\hspace{2.0mm}=\ul{0}10101&\},\\
S_4&=&\{\xi_4^1\hspace{2.0mm}=00011\ul{1},&\xi_4^2\hspace{2.0mm}=0010\ul{1}1,&\xi_4^3\hspace{2.0mm}=01\ul{0}011,&\xi_4^4=\ul{0}10101\},
\end{array}\end{eqnarray}

\noindent (based on \cite[display~(4.2)]{Hcs}) where $w$ is any (possibly empty) Dyck word. Consider also the sets $\ul{S}_1(w)$, $\ul{S}_2$, $\ul{S}_3$, $\ul{S}_4$ obtained respectively from $S_1(w)$, $S_2$, $S_3$, $S_4$ by having their component Dyck paths
$\ul{\xi}_{1^w}^j$, $\ul{\xi}_2^j$, $\ul{\xi}_3^j$, $\ul{\xi}_4^j$
defined as the complements of the reversed strings of the corresponding Dyck paths
$\xi_{1^w}^j$, $\xi_2^j$, $\xi_3^j$, $\xi_4^j$.
Note that each Dyck word in the subsets of display~(\ref{!}) has just one underlined entry. By denoting
\begin{eqnarray}\label{denot}\xi_{1^w}^j=x_sx_{s-1}\cdots x_2x_1x_0\;\mbox{ and }\;\xi_i^j=x_sx_{s-1}\cdots x_2x_1x_0,\mbox{ for }i=2,3,4,\end{eqnarray}
where $j=1,2,3$ for $i=1,2,3$ and $j=1,2,3,4$ for $i=4$ and adequate $s$ in each case, the barred positions in~(\ref{!}) are the targets of the following correspondence $\Phi$:
\begin{eqnarray}\label{!!}\begin{array}{llll}
\Phi(\xi_{1^w}^1)=1,&\Phi(\xi_{1^w}^2)=4,&\Phi(\xi_{1^w}^3)=0,&\\
\Phi(\xi_2^1)\hspace{2.0mm}=6,&\Phi(\xi_2^2)\hspace{2.0mm}=0,&\Phi(\xi_2^3)\hspace{2.0mm}=2,&\\
\Phi(\xi_3^1)\hspace{2.0mm}=0,&\Phi(\xi_3^1)\hspace{2.0mm}=1,&\Phi(\xi_3^3)\hspace{2.0mm}=5,&\\
\Phi(\xi_4^1)\hspace{2.0mm}=0,&\Phi(\xi_4^2)\hspace{2.0mm}=1,&\Phi(\xi_4^3)\hspace{2.0mm}=3,&\Phi(\xi_4^4)=5.\\
\end{array}\end{eqnarray}
The correspondence $\Phi$ is extended over the Dyck words $\ul{\xi}_{1^w}^j$, $\ul{\xi}_2^j$, $\ul{\xi}_3^j$, $\ul{\xi}_4^j$ with their barred positions taken reversed with respect to the corresponding barred positions in $\xi_{1^w}^j$, $\xi_2^j$, $\xi_3^j$, $\xi_4^j$, respectively.

Adapting from \cite{Hcs}, we define an hypergraph $H_k$ with $V(H_k)$ formed by the lists $L(i)$ with $b(i)\in V(T)\cap b([0,{\mathcal C}_k-1])$ and as hyperedges the subsets

$$\{L(i_j)|j\in\{1,2,3\}\}\subset V(H_k)\mbox{ and }\{L(i_j)|j\in\{1,2,3,4\}\}\subset V(H_k)$$

\noindent whose member lists $L(i_j)$ have the Dyck words $f(i_j)$ corresponding to their initial rows, the Dyck nests $F(i_j)$, for $j=1,2,3\mbox{ or }j=1,2,3,4$, containing respective Dyck subwords as members in the collection (indexed by $j$)
$$\{\xi_{1^w}^j, \xi_2^j, \xi_3^j, \xi_4^j, \ul{\xi}_{1^w}^j, \ul{\xi}_2^j, \ul{\xi}_3^j, \ul{\xi}_4^j\}$$ in the same 6 or 8 fixed positions $x_i$ (for specific indices $i\in\{0,1,\ldots,s\}$ in~(\ref{denot})) and forming respective subsets
$$\{\xi_{1^w}^j|j=1,2,3\},\; \{\ul{\xi}_{1^w}^j|j=1,2,3\},\; \{\xi_4^j|j=1,2,3,4\},\; \{\ul{\xi}_4^j|j=1,2,3,4\},$$
$$\{\xi_i^j|j=1,2,3\}\;\mbox{ and }\;\{\ul{\xi}_i^j|j=1,2,3\},\;\mbox{ for both }i=2\mbox{ and }3.$$

\begin{example}\label{ex6} Two hyperedges $h_0,h_1$ of $H_3$ are obtained from the six indices $z\in\{0,1\}$ in items (i)-(iii) (see Example~\ref{k=3-310}) of corresponding superindices in Table~\ref{uniform}. These two hyperedges are obtained from the triples $\tau_0,\tau_1$ in Example~\ref{k=3-310}, with substrings $\xi_{1(\epsilon)}^j$ or $\xi_i^j$ ($j=3,4$) realized via the last six entries of the initial lines $F(i)$ of the lists $L(i)$ ($i\in[0,4]$), where the involved flippable tuples have those six entries overlined in the lower one of its two component lines. For $h_0$ (resp., $h_1$), $f(0)$, $f(1)$, $f(2)$ (resp., $f(0)$, $f(3)$, $f(4)$) are Dyck words
$\xi_{1^\epsilon}^1$, ${\ul{\xi}}_3^2$, $\xi_{1^\epsilon}^3$ (resp., $\xi_{1^\epsilon}^3$, $\xi_3^2$, $\xi_4^3$), with their entries in positions $\Phi(\xi_{1^\epsilon}^1)=1$, $\Phi({\ul{\xi}}_4^2)=4$, $\Phi(\xi_{1^\epsilon}^3)=0$ (resp., $\Phi(\xi_{1^\epsilon}^3)=0$, $\Phi(\xi_3^2)=1$, $\Phi(\xi_4^3)=5$) in thick trace. Then $H_3$ contains the connected subhypergraph $H_3'$ formed by $h_0,h_1$, that are both incident to the vertex $L(0)$ of $H_3$. This is used to construct the Hamilton cycle mentioned in Example~\ref{k=3-310}; $H'_3$ is schematically represented over the lower right in Table~\ref{dia26}.  \end{example}

\begin{example}\label{ex7} For $k=4$,
in a likewise manner to that of Example~\ref{ex6}, Table~\ref{dia26} shows
the 14 columns $L(i)$, for $i\in[0,12]$, where we use hexadecimal notation with $a=10,b=11,c=12,d=13$. The information contained in the table leads to a subhypergraph $H'_4$ of $H_4$ represented
in its lower-right corner with the hyperedges
$$h_0=(0,2,a),\; h_1=(8,7,5),\; h_2=(7,6,a),\; h_3=(1,4,6),\; h_4=(1,9,d),\; h_5=(3,b,c,d).$$
The respective triples of Dyck words $\xi_{1^w}^j$ or $\ul{\xi}\,_{1^w}^j$ or $\xi_i^j$ or $\ul{\xi}\,_i^j$ ($j=2,3,4$) may be expressed as follows by replacing the Greek letters $\xi$ by the values of the correspondence $\Phi$:
$$
(\ul{5}_3^1,\ul{4}_3^2,\ul{0}_3^3),\;
(6_2^1,0_2^2,2_2^3), \;
(1_{1^{01}}^1,4_{1^{01}}^2,0_{1^{01}}^3),\;
(1_{1^\epsilon}^1,4_{1^\epsilon}^2,0_{1^\epsilon}^3),\;
(0_3^1,1_3^2,5_3^3),\;
(0_4^1,1_4^2,3_4^3,5_4^4),$$

\begin{table}[htp]
$$\begin{array}{|l|l|l|l|l|}\hline
^{k=4}_{i=0}&^{k=4}_{i=1}&^{k=4}_{i=2}&^{k=4}_{i=3}&^{k=4}_{i=4}\\
00\ol{{\ul0}00111}1|\ol{\xi}_4^1&0\ol{00011{\ul1}}01|\xi_4^1&00\ol{0{\ul0}1101}1|\xi_{1^\epsilon}^2& 001\ol{00011{\ul1}}|\xi_4^1&0\ol{0{\ul0}1101}01|\xi_{1^\epsilon}^2\\
                 &0\ol{0001{\ul1}1}01|\xi_{1^\epsilon}^1                 &                                                     &                                           &                                              \\\hline
0{\ul1}{2}344321|_{1000}\!&02344{3}{2}{\ul1}1|_{7101}\!&0{\ul1}3{4}43221|_{1202}\!&022{\ul1}3443{1}|_{3303}\!&03{4}4322{\ul1}1|_{7404}\\
4332101{\ul2}4|_{7550}&43210123{\ul4}|_{8041}&433102{\ul2}14|_{5b42}&4213310{\ul2}4|_{7863}&43102213{\ul4}|_{8334}\\
110{\ul3}44322|_{3420} &1{\ul2}3443210|_{1081}&110{\ul3}32442|_{3922}&14410{\ul3}322|_{5c73}&13443{\ul2}210|_{5284}\\
32211{\ul0}443|_{5d50}&32210{\ul1}344|_{5241}&32\ol{21{\ul4}410}3|_{4c72}^{0\uparrow 0}&310221{\ul4}43|_{6c23}&321012{\ul3}44|_{6034}\\
2211{\ul0}4433|_{4d40}&2\ol{21{\ul4}4{3}3}10|_{3b81}^{3\uparrow 6}&22{\ul1}103443|_{2442}&2442{\ul1}1033|_{4963}&2{\ul3}4432110|_{1184}\\
211044{\ul3}32|_{6d30}&2110{\ul3}3244|_{4931}&2134431{\ul0}2|_{7372}&21013443{\ul2}|_{8223}&211{\ul0}24433|_{3634}\\
33{\ul2}210144|_{2a50}&33{\ul2}214410|_{2c81} &332101{\ul2}44|_{6542}&3{\ul4}4\ol{321012}|_{1063}^{5\uparrow b}&33{\ul2}442110|_{2984}\\
10\ol{234432}{\ul1}|_{8110}^{0\downarrow a}&102442{\ul1}33|_{6811}&10344322{\ul1}|_{8412}&10{\ul1}234432|_{2013}&1\ol{022{\ul1}44}33|_{4b14}^{3\uparrow 1}\\
{\ul4}43210123|_{0050}&{\ul4}\ol{421013}32|_{0741}^{4\uparrow d}&{\ul4}43102213|_{0342}&{\ul4}43321012|_{0563}&{\ul4}41033221|_{0c34}\\\hline
^{k=4}_{i=5}&^{k=4}_{i=6}&^{k=4}_{i=7}&^{k=4}_{i=8}&^{k=4}_{i=9}\\
0\ol{00010{\ul1}11}|\xi_2^3&0\ol{00101{\ul1}}01|\xi_{1^\epsilon}^3& 0\ol{0010011{\ul1}}|\xi_2^2& 0\ol{0{\ul0}110011}|\xi_2^1& 0\ol{0100{\ul1}1}01|\xi_3^2\\
                  &                        0\ol{001{\ul0}1101}|{\ul\xi}_{1^{01}}^3& 0\ol{001001{\ul1}1}|\xi_{1^{01}}^1&      &\\\hline
0{\ul1}2443{3}21|_{1505}&0244{3}3{2}{\ul1}1|_{7606}&0{\ul1}33244{2}{1}|_{1707}&02{4}42{\ul1}331|_{5808}&03324{4}2{\ul1}1|_{7909}\\
4332110{\ul2}4|_{7665}&43211023{\ul4}|_{8156}&433{\ul2}21014|_{3a67}&4210133{\ul2}4|_{7738}&43221013{\ul4}|_{8259}\\
11044{\ul3}322|_{5d25}&1{\ul2}4433210|_{1586}&11{\ul0}234432|_{2127}&134431{\ul0}22|_{6368}&133{\ul2442}10|_{3789}\\
322101{\ul4}43|_{6a45}&322{\ul1}10344|_{3456}&3244210{\ul1}3|_{7767}&32101244{\ul3}|_{8538}&31022{\ul1}344|_{5329}\\
2\ol{21{\ul4}41033}|_{3c65}^{1\uparrow 8}&22{\ul1}344310|_{2386}&2\ol{2101{\ul4}433}|_{5a37}^{2\uparrow a}&2{\ul3}4432101|_{1078}&2\ol{442{\ul1}{3}3}10|_{4889}^{4\uparrow 1}\\
2110{\ul3}4432|_{4435}&21331{\ul0}244|_{5856}&214410{\ul3}32|_{6c47}&211{\ul0}23443|_{3138}&210133{\ul2}44|_{6729}\\
33{\ul2}211044|_{2d65}&3310{\ul2}4421|_{4846}&3310{\ul2}2144|_{4b37}&33{\ul2}442101|_{2778}&3{\ul4}4310221|_{1359}\\
10244332{\ul1}|_{8615}&1\ol{03322{\ul1}44}|_{6c16}^{2\uparrow 8}&10332442{\ul1}|_{8c17}&1\ol{022{\ul1}3443}|_{4318}^{1\downarrow 7}&10{\ul1}344322|_{2219}\\
{\ul4}43211023|_{0155}&{\ul4}\ol{421103}32|_{0956}^{3\downarrow 4} &{\ul4}\ol{43221013}|_{0267}^{1\uparrow 5}&{\ul4}42133102|_{0878}&{\ul4}43310221|_{0b59}\\\hline
^{k=4}_{i=a}&^{k=4}_{i=b}&^{k=4}_{i=c}&^{k=4}_{i=d}&H'_3:\\
0\ol{0010101{\ul1}}|\xi_{1^{01}}^3&  001\ol{0010{\ul1}1}|\xi_4^2&001\ol{01{\ul0}011}|\xi_4^3 &  0\ol{{\ul0}10101}01|\xi_4^4&\hspace*{10mm}^{1\hspace{9mm}3}_{|\;>\;0\;<\;|}\\
00\ol{01010{\ul1}}1|{\ul\xi}_4^4&                                &                           &                              001\ol{{\ul0}10101}|\xi_4^4& \hspace*{10mm}^{2\hspace{9mm}4}\\\hline
0{\ul1}44332{2}{1}|_{1a0a}&022{\ul1}443{3}1|_{3b0b}&03322{\ul{1}}441|_{5c0c}&0{4}4{3}322{\ul1}1|_{7d0d}&H'_4:\\
433{\ul2}21104|_{3d7a}&41033{\ul2}214|_{5c2b}&421\ol{1033{\ul2}4}|_{794c}^{5\uparrow d}&43221103{\ul4}|_{846d}&\\
12{\ul0}244331|_{262a}&1331{\ul0}2442|_{484b}&144331{\ul0}22|_{6b6c}&14433{\ul2}210|_{5a8d}&                        \hspace*{10mm}^5_{|\;>\;8}\\
32442{\ul1}103|_{597a}&3102442{\ul1}3|_{785b}&32110244{\ul3}|_{864c}&321102{\ul3}44|_{614d}&                         \hspace*{10mm}^{7\hspace{9mm}0}_{|\;>\;a\;<\;|}\\
2210{\ul1}3443|_{423a}&244\ol{210{\ul1}33}|_{675b}^{5\uparrow c}&244{\ul3}32101|_{357c}&244{\ul3}32110|_{368d}& \hspace*{10mm}^{6\hspace{9mm}2}_{|\;>\;4}\\
2144331{\ul0}2|_{7b7a}&21014433{\ul2}|_{8a2b}&2101{\ul2}3443|_{402c}&210\ol{1{\ul2}4433}|_{452d}^{5\uparrow c}& \hspace*{10mm}^1_{|\;>\;9}\\
33\ol{2110\ul{2}4}4|_{665a}^{0\uparrow 2}&3{\ul4}4321102|_{117b}&3{\ul4}4322101|_{127c}&3{\ul4}4322110|_{148d}&  \hspace*{10mm}^{d\;-\;3}_{|\hspace*{5mm}|}\\
10443322{\ul1}|_{8d1a}&10{\ul1}244332|_{251b}&10{\ul1}332442|_{271c}&1\ol{0{\ul1}4433}22|_{2a1d}^{4\uparrow 9}& \hspace*{10mm}^{c\;-\;b}\\
{\ul4}\ol{43221103}|_{047a}^{2\uparrow 6}&{\ul4}43321102|_{067b}&{\ul4}43322101|_{0a7c}&{\ul4}43322110|_{0d8d}&\\\hline
\end{array}$$
\caption{Uniform 2-factors provided by permutations $\pi$ for $k=4$.}
\label{dia26}
\end{table}

\noindent where we can also write $(\ul{5}_3^1,\ul{4}_3^2,\ul{0}_3^3)=\ul{(0_3^1,1_3^2,5_3^3)}$.
 The flippable tuples $FT(i,j)$ allow to compose five flipping $6$-cycles and one flipping $8$-cycle, allowing to integrate by symmetric differences a Hamilton cycle in $O_4$.

We represent $H_k$ as a simple graph $\psi(H_k)$ with $V(\psi(H_k))=V(H_k)$ by replacing each hyperedge $e$ of $H_k$ by the clique $K(e)=K(V(e))$ so that $\psi(H_k[e])=K(e)$, being such replacements the only source of cliques of $\psi(H_k)$. A {\it tree} $T$ of $H_k$ is a subhypergraph of $H_k$ such that: {\bf(a)} $\psi(T)$ is a connected union of cliques $K(V(e))$; {\bf(b)} for each cycle $C$ of $\psi(H_k)$, there exist a unique clique $K(V(e))$ such that $C$ is a subgraph of $K(e)$. A {\it spanning tree} $T$ of $H_K$ is a tree of $H_k$ with $V(T)=V(H_k)$. Clearly, the subhypergraphs $H'_k$ of $H_k$  for $k=3$ and 4 are corresponding spanning trees.

A subset $G$ of hyperedges of $H_k$ is said to be {\it conflict-free} \cite{Hcs} if: {\bf(a)} any two hyperedges of $G$ have at most one vertex in common; {\bf(b)} for any two hyperedges $g,g'$ of $G$ with a vertex in common, the corresponding images by $\Phi$ (as in display~(\ref{!!})) in $g$ and $g'$ are distinct. A proof of the following final result is included, as our viewpoint and notation differs from that of its proof in \cite{Hcs}.
\end{example}

\begin{theorem}\label{L6} A conflict-free spanning tree of $H_k$ yields a Hamilton cycle of $O_k$, for every $k\ge 3$. Moreover, distinct conflict-free spanning trees of $H_k$ yield distinct Hamilton cycles of $H_k$, for every $k\ge 6$.
 \end{theorem}

\begin{proof} Let $D_k$ be the set of all Dyck words of length $2k$ and, recalling display (\ref{!}), let
\begin{eqnarray}\label{rec1}\begin{array}{lll}
E_2=\{0101\}, \!&\! E_3=S_4,\!&\! E_k=01D_{k-1}, \forall k>3,\\
F_2=\{0011\}, \!&\! F_3=D_3\setminus E_3=\{001101\}.					     \!&\!F_k=D_k\setminus 01D_{k-1}, \forall k>3.
\end{array}\end{eqnarray}
In particular, $0101(01)^{k-2}\in E_k$ and $0011(01)^{k-2}\in F_k$. Now, let
\begin{eqnarray}\label{rec2}\begin{array}{llll}
{\mathcal E}_2=\emptyset, & \mathcal{E}_3=\{S_4\},       & \mathcal{T}_3=\{S_1(\epsilon),S_3\},&\mathcal{E}_k=01\mathcal{T}_{k-1},  \forall k>3,\\
{\mathcal F}_2=\emptyset, & \mathcal{F}_3=\emptyset, & \mathcal{F}_4=\{S_1(01),S_2,0\ul{S}_31,S_1(\epsilon)01\}.&
\end{array}\end{eqnarray}
Let us set $\mathcal{F}_k$ as a function of $\mathcal{E}_2,\ldots,\mathcal{E}_{k-1},\mathcal{F}_2,\ldots,\mathcal{F}_{k-1},\mathcal{T}_{k-2}$, as follows:
For $1<j\le k$, let $F_k^j=\cup_{i=2}^j\{0\ul{u}1v|u\in D_{i-1},v\in D_{k-1}\}$. Since $F_k=F_k^k$, then the following implies the existence of a spanning tree of $H_k[F_k]$.

\begin{lemma} For every $1<j\le k$, there exists a spanning tree $\mathcal{F}_k^j$ of $H_k[F_k^j]$.
\end{lemma}

\begin{proof}
Lemma 7 \cite{Hcs} asserts that if $\tau$ is a flippable tuple and $u,v$ are Dyck words, then: {\bf(i)} $u\tau v$ is a flippable tuple if $|u|$ is even; {\bf(ii)} $u\ul{\tau}v$ is a flippable tuple if $|u|$ is odd. Lemma 8 \cite{Hcs} insures that the collections in (\ref{!}) are flippable tuples. Using those two lemmas of \cite{Hcs}, we define
$\Psi$ as the set of all the flippable tuples $u\tau v$ and $u\ul{\tau}v$ arising from (\ref{!}).
Moreover, we define $\Psi_2=\emptyset$ and $\Psi_{k}=\Psi\cap D_{k}$, for ${k}>2$.

Since $F_{k}^2=0011D_{{k}-2}$, we let $\mathcal{F}_{k}^2=0011\mathcal{T}_{{k}-2}$. Assuming $2<j\le{k}$, since $D_{j-2}=E_{j-1}\cup F_{j-1}$ is a disjoint union, then we have the following partition:
\begin{eqnarray}\label{quelio}F_{k}^j=F_{k}^{j-1}\cup_{v\in D_{{k}-j} }(0\ul{D}_{j-1}1v)=F_{k}^{j-1}\cup_{v\in D_{{k}-j} }((0\ul{E}_{j-1}1v)\cup(0\ul{F}_{j-1}1v)).\end{eqnarray}
For every $v\in D_{{k}-j}$, the elements of $\tau(v)=S_1((01)^{j-3})v\in\Psi_{k}$, are:
\begin{eqnarray}\label{rec3}\begin{array}{ccc}
0(01)^{j-3}001\ul{1}1v\in 0\ul{F}_{j-1}1v, & 0(01)^{j-3}0101\ul{1}v\in 0\ul{E}_{j-1}1v, & 0(01)^{j-3}\ul{0}1101v\in F_{k}^{j-1}.\end{array}\end{eqnarray}
Now, we let
\begin{eqnarray}\label{formula}\mathcal{F}_{k}^j=\mathcal{F}_{k}^{j-1}\cup(\cup_{v\in D_{{k}-j}}(\{\tau(v)\}\cup (0\ul{\mathcal{E}}_{j-1}1v)\cup (0\ul{\mathcal{F}}_{j-1}1v))),\end{eqnarray} which defines a spanning tree of $H_{k}[F_{k}^j]$. \end{proof}

Now, the elements of  $\tau=S_3(01)^{{k}-3}\in\Psi_{k}$ are:
\begin{eqnarray}\label{rec4}\begin{array}{ccc}
00011\ul{1}(01)^{{k}-3}\in F_{k}$, (${k}>3), & 0100\ul{1}1(01)^{{k}-3}\in 01E_{{k}-1}, & \ul{0}10101(01)^{{k}-3}\in 01F_{{k}-1}.\end{array}\end{eqnarray}

The sets $F_{k}$, $01E_{{k}-1}$ and $01F_{{k}-1}$ form a partition of $D_{k}$. We take the spanning trees of the subhypergraphs induced by these three sets and connect them into a single spanning tree of $H_{k}$ by means of the triple $\tau$, that is:
\begin{eqnarray}\label{rec5}H'_{k}=\mathcal{F}_{k}\cup\{\tau\}\cup 01\mathcal{E}_{{k}-1}\cup 01\mathcal{F}_{{k}-1}.\end{eqnarray}\end{proof}

\begin{example} Example~\ref{ex6} uses $H'_3$ in display~(\ref{rec2}), with $S_1(\epsilon)=012$ and $S_3=034$ yielding the hypergraph $H'_3$ depicted over the lower-right enclosure of Table{dia26}. Example~\ref{ex7} uses $H'_{k}$ in display~(\ref{rec5}) for ${k}=4$, $\mathcal{F}_4$ and $\mathcal{E}_3$ in display~(\ref{rec2}) and $\tau$ in display~(\ref{rec4}), with $S_1(01)=67a$, $S_2=875$, $0\ul{S}_31=02a$, $S_1(\epsilon)=146$, being these four triples the elements in $\mathcal{F}_4$; $01S_4=3bcd$, this one as the only element of $01\mathcal{E}_3$, (while $\mathcal{F}_3=\emptyset$); and $\tau=02a$, yielding the hypergraph $H'_4$ depicted in the lower-right enclosure of Table~\ref{uniform}.\end{example}

\begin{corollary}\label{the-end}  To each Hamilton cycle in $O_k$ produced by {\rm Theorem~\ref{L6}} corresponds a Hamilton cycle in $M_k$.
\end{corollary}

\begin{proof}
For each vertical list $L(i)$, let $L^M(i)$ be a corresponding vertical list in $M_k$ which is obtained from $L(i)$. Then, Theorem~\ref{L6} can be adapted to producing Hamilton cycles in the $M_k$ by repeating the argument in its proof in replacing the lists $L(\alpha)$ by lists $L^M(\alpha)$, since they have locally similar behaviors, being the cycles provided by the lists $L^M(\alpha)$ twice as long as the corresponding lists $L(\alpha)$, so the said local behavior happens twice around opposite (rather short) subpaths. Combining Dyck-word triples and quadruples as in display~(\ref{!}) into adequate pullback liftings (of the covering graph map $M_k\rightarrow O_k$ in the lists $L^M(\alpha)$ of those parts of the lists $L(\alpha)$ in which the necessary symmetric differences take place to produce the Hamilton cycles in $O_k$ will produce corresponding Hamilton cycles in $M_k$.
\end{proof}

\end{document}